\documentclass[leqno,a4paper,10pt]{amsart}

\usepackage{amsmath, amsfonts, amsthm, amssymb, dsfont}

\usepackage{times}
\usepackage{microtype}

\usepackage{cite}

\usepackage{mathrsfs}
\usepackage{enumerate, xspace}
\usepackage{graphicx}
\usepackage{color}
\usepackage[marginpar]{todo}

\usepackage[all, knot]{xy}
\xyoption{arc}

\usepackage{pgf,tikz}
\usepackage{mathrsfs}
\usetikzlibrary{arrows}

\usepackage{verbatim}

\usepackage{caption}
\usepackage{subcaption}

\usepackage[
	pdfauthor={R\'emi Coulon and Markus Steenbock},
	pdftitle={Product set growth in Burnside groups},
	pdfkeywords={product sets, small tripling,  growth, hyperbolic groups, hyperbolic space, acylindrical actions, small cancellation theory, infinite periodic groups, Burnside problem},
	pdftex
]{hyperref}

\hypersetup{
	colorlinks,%
	citecolor=black,%
	filecolor=black,%
	linkcolor=black,%
	urlcolor=black,%
}

\theoremstyle{plain}

\newtheorem{thm}{Theorem}[section]
\newtheorem{prop}[thm]{Proposition}
\newtheorem{lem}[thm]{Lemma}		
\newtheorem{cor}[thm]{Corollary}

\theoremstyle{definition}		

\newtheorem{df}[thm]{Definition}
\newtheorem{ex}[thm]{Example}	
				
\newtheorem{rem}[thm]{Remark}

\theoremstyle{remark}

\DeclareMathOperator{\diam}{diam}

\DeclareMathOperator{\length}{length}

\newcommand{\cX}{\dot{X}}
\newcommand{\qX}{\overline{X}}
\newcommand{\qG}{\overline{G}}

\DeclareMathOperator{\Fix}{Fix}

\newcommand{\q}[1]{\overline{#1}}

\newcommand{\stab}[1]{\operatorname{Stab}(#1)}

\renewcommand{\leq}{\leqslant}
\renewcommand{\geq}{\geqslant}

\title{Product set growth in Burnside groups}
\date{}
\author{R\'emi Coulon}
\address{IRMAR\\ Univ Rennes et CNRS \\  35000 Rennes \\ France}
\email{remi.coulon@univ-rennes1.fr}
\author{Markus Steenbock }
\address{Fakult\"at f\"ur Mathematik \\ Universit\"at Wien \\  1090 Wien  \\ Austria}
\email{markus.steenbock@univie.ac.at}

\date{\today}


\begin{document}

\begin{abstract}
Given a periodic quotient of a torsion-free hyperbolic group, we provide a fine lower estimate of the growth function of any sub-semi-group.
This generalizes results of Razborov and Safin for free groups.
\end{abstract}

\maketitle

\section{Introduction}
If $V$ is a subset in a group $G$, we denote by $V^r\subset G$ the set of all group elements that are represented by a product of exactly $r$ elements of $V$.  
In this paper we are interested in the growth of $V^r$.
Such a problem has a long history which goes back (at least) to the study of additive combinatorics. See for instance \cite{nathanson_additive_1996, tao_additive_2006}.
In the context of non-abelian groups, it yields to the theory of approximate subgroups, see \cite{terence_product_2008,breuillard_structure_2012}, and relates to spectral gaps in linear groups, see 
\cite{helfgott_growth_2008,bourgain_spectral_2008,bourgain_spectral_2012}, as well as exponential growth rates of negatively curved groups, \cite{koubi_croissance_1998,arzhantseva_lower_2006,breuillard_joint_2018,fujiwara_rates_2020}. 

 If $G$ is a free group, Safin \cite{safin_powers_2011}, improving former results by Chang \cite{chang_product_2008} and Razborov \cite{razborov_product_2014}, proves that there exists $c > 0$ such that for every finite subset $V \subset G$, either $V$ is contained in a cyclic subgroup, or for every $r \in \mathbb N$, we have
 \begin{equation*}
 	|V^r| \geq (c |V|)^{[(r+1)/2]}.
 \end{equation*}
This estimate can be thought of as a quantified version of the Tits alternative in $G$.
 A similar statement holds for ${\rm SL}_2(\mathbb{Z})$ \cite{chang_product_2008}, free products, limit groups \cite{button_explicit_2013} and groups acting on $\delta$-hyperbolic spaces \cite{delzant_product_2020}. All these groups display strong features of negative curvature, inherited from a non-elementary acylindrical action on a hyperbolic space.
Some results are also available for solvable groups \cite{tao_freiman_2010,button_explicit_2013}, as well as mapping class groups and right-angled Artin groups \cite{kerr_product_2021}.

By contrast, in this work, we focus on a class of groups which do not admit any non-elementary action on a hyperbolic space, namely the set of infinite groups with finite exponent, often referred to as of Burnside groups.

\subsection{Burnside groups of odd exponent}
Given a group $G$ and an integer $n$, we denote by $G^n$ the subgroup of $G$ generated by all its $n$-th powers.
We are interested in quotients of the form $G/G^n$ which we call \emph{Burnside groups} of exponent $n$.
If $G = \mathbb F_k$ is the free group of rank $k$, then $B_k(n) = G/G^n$ is the \emph{free Burnside group} of rank $k$ and exponent $n$.
The famous Burnside problem asks whether a finitely generated free Burnside group is necessarily finite.

Here, we focus on the case that the exponent $n$ is odd. 
By Novikov's and Adian's solution of the Burnside problem, it is known that $B_k(n)$ is infinite provided $k \geq 2$ and $n$ is a sufficiently large odd integer \cite{adian_burnside_1979}. See also \cite{olshanksii_novikov_1982,delzant_courbure_2008}.
More generally, if $G$ is a non-cyclic, torsion-free, hyperbolic group, then the quotient $G/G^n$ is infinite provided $n$ is a sufficiently large odd exponent \cite{olshanskii_periodic_1991,delzant_courbure_2008}.
Our main theorem extends Safin's result to this class of Burnside groups of odd exponents.

\begin{rem}
	Note that free Burnside groups of sufficiently large even exponents are also infinite.
	This was independently proved by Ivanov \cite{Ivanov:1994kh} and Lysenok \cite{Lysenok:1996kw}.
	Moreover any non-elementary hyperbolic group admits infinite Burnside quotients, see \cite{Ivanov:1996va,coulon_infinite_2018}.
	Nevertheless in the remainder of this article we will focus on torsion-free hyperbolic groups and odd exponents.
	In Section~\ref{sec: discussion-even-exp} we discuss the difficulties to extend our results to the case of even exponents.
\end{rem}

\begin{thm}\label{T: intro - large powers free burnside}
Let $G$ be a non-cyclic, torsion-free hyperbolic group.
There are numbers $n_0>0$ and $c >0$  such that for all odd integer $n\geqslant n_0$ the following holds.
Given a finite subset $V \subset G/G^n$, either $V$ is contained in a finite cyclic subgroup, or for all $r\in \mathbb N$, we have
$$ |V^r|\geqslant \left( c |V| \right)^{[(r+1)/2]}.$$
\end{thm}

Observe that the constant $c$ only depends on $G$ and not on the exponent $n$. 
Recall that Burnside groups do not act, at least in any useful way, on a hyperbolic space. 
Indeed, any such action is either elliptic or parabolic. 
On the other hand, it is well-known that any linear representation of a finitely generated Burnside group has finite image.
Thus our main theorem is not a direct application of previously known results.

Let us mention some consequences of Theorem~\ref{T: intro - large powers free burnside}.
If $V$ is a finite subset of a group $G$, one defines its \emph{entropy} by 
\begin{equation*}
	h(V)= \limsup_{r \to \infty} \frac{1}{r}\log |V^r|.
\end{equation*}
The group $G$ has \emph{uniform exponential growth} if there exists $\varepsilon > 0$ such that for every finite symmetric generating subset $V$ of $G$, $h(V) > \varepsilon$.
In addition, $G$ has \emph{uniform uniform exponential growth} if there exists $\varepsilon > 0$ such that for every finite symmetric subset $V \subset G$, either $V$ generates a virtually nilpotent group, or $h(V) > \varepsilon$.

\begin{cor}
\label{C: intro - large powers free burnside} 
Let $G$ be a non-cyclic, torsion-free hyperbolic group.
There are numbers $n_0>0$ and $\alpha>0$ such that for all odd integer $n \geqslant n_0$, the following holds.
Given a finite subset $V \subset G/G^n$ containing the identity, either $V$ is contained in a finite cyclic subgroup, or
$$ h(V) \geq \alpha \ln |V| \geq \alpha \ln 3.$$
In particular, $G/G^n$ has uniform uniform exponential growth.
\end{cor}

It was already known that free Burnside groups of sufficiently large odd exponent have uniform exponential growth, see Osin~\cite[Corollary 1.4]{osin_uniform_2007} and Atabekyan~\cite[Corollary 3]{atabekyan_uniform_2009}.
Note that Theorem~2.7 in \cite{osin_uniform_2007} actually shows that free Burnside groups have uniform uniform exponential growth.
Nevertheless, to the best of our knowledge, the result was not proved for Burnside quotients of hyperbolic groups.
We shall also stress the fact that, unlike in Corollary~\ref{C: intro - large powers free burnside}, the growth estimates provided in \cite{osin_uniform_2007, atabekyan_uniform_2009} depend on the exponent $n$.
The reason is that the parameter $M$ given for instance by \cite[Theorem~2.7]{osin_uniform_2007} is a quadratic function of $n$.

Given a group $G$ with uniform exponential growth, a natural question is whether or not there exists a finite generating set that realizes the minimal growth rate.
The first inequality is a statement \emph{\`a la} Arzhantseva-Lysenok for torsion groups, see \cite[Theorem~1]{arzhantseva_lower_2006}.
The philosophy is the following: if the set $V$ has a small entropy, then it cannot have a large cardinality.
In particular, if we expect the minimal growth rate to be achieved, we can restrict our investigation to generating sets with fixed cardinality.
Note that this is exactly the starting point of the work of Fujiwara and Sela in the context of hyperbolic groups, \cite{fujiwara_rates_2020}.

Let us discuss now the power arising in Theorem \ref{T: intro - large powers free burnside}.
We claim that, as our estimate is independent of the exponent $n$, the power $(r+1)/2$ is optimal.
For this purpose we adapt an example of \cite{safin_powers_2011}. 

\begin{ex} \label{E: optimal} 
Let $g$ and $h$ be two elements in $B_2(n)$ such that $g$ generates a group of order $n$, that does not contain $h$. 
Consider the set 
\begin{equation*}
	V_N=\left\{1, g,  g^2 ,\ldots , g^N, h\right\}.
\end{equation*}
Whenever the exponent $n$ is sufficiently large compared to $N$, we have $|V_N^r|\sim N^{[(r+1)/2]}$ while $|V_N|=N+1$. 
\end{ex}

Button observed the following fact.
Assume that there is $c>0$ and $\varepsilon>0$ with the following property: for all finite subsets $V$ in a group $G$ that are not contained in a virtually nilpotent subgroup, we have $|V^3|\geqslant c |V|^{2+\varepsilon}$.
Then $G$ is either virtually nilpotent, or of bounded exponent \cite[Proposition 4.1]{button_explicit_2013}.
We do not know if such a non-virtually nilpotent group exists.

\subsection{Groups acting on hyperbolic spaces} 
 In the first part of our paper, we revisit product set growth for a group $G$ acting on a hyperbolic space $X$, see \cite[Theorem 1.14]{delzant_product_2020}. For this purpose, we use the notion of an acylindrical action, see \cite{sela_acylindricity_1997,bowditch_tight_2008}. 
Given a subset $U$ of $G$, we exploit its \emph{$\ell^\infty$-energy} defined as
\begin{equation*}
	\lambda(U) = \inf_{x \in X}\sup_{u \in U} |ux - x|.
\end{equation*}

\begin{rem}
	Unlike in \cite{delzant_product_2020}, we will not make use of the $\ell^1$-energy.
	Our motivation is mostly technical.
	We explain this choice in Section~\ref{sec: strategy for burnside groups}.
\end{rem}

 \begin{thm}[see Theorem~\ref{T: main acylindrical}]
 \label{T: main acylindrical - intro}
	Let $G$ be group acting acylindrically on a hyperbolic length space $X$.
	There exists a constant $C > 0$ such that for every finite subset $U\subset G$ with $\lambda(U)> C$, 
	\begin{enumerate}
		\item either $\displaystyle |U| \leq  C$,
		\item or there is subset $W\subset U^2$ freely generating a free sub-semigroup of cardinality 
		\begin{equation*}
			|W| \geq \frac 1C\frac 1{\lambda(U)} |U|.
		\end{equation*} 
	\end{enumerate}
\end{thm}

\begin{rem}
	For simplicity we stated here a weakened form of Theorem~\ref{T: main acylindrical}.
	Actually we prove that the constant $C$ only depends on the hyperbolicity constant of the space $X$ and the acylindricity parameters of the action of $G$.
	The set $W$ is also what we called \emph{strongly reduced}, see Definition~\ref{def: reduced set}.
	Roughly speaking this means that the orbit map from the free semi-group $W^*$ to $X$ is a quasi-isometric embedding.
\end{rem}

There is quite some literature on finding free sub-semigroups in powers of symmetric subsets $U$ in groups of negative curvature, see \cite{koubi_croissance_1998,arzhantseva_lower_2006,breuillard_joint_2018}. 
We can for example compare Theorem \ref{T: main acylindrical} to Theorem 1.13 of \cite{breuillard_joint_2018}. In this theorem, under the additional assumption that $U$ is symmetric, the authors construct a $2$-element set in $U^r$ that generates a free sub-semigroup, where the exponent $r$ does only depend on the doubling constant of the space. 
Let us highlight two important differences.
First we do not assume that the set $U$ is symmetric. 
In particular, we cannot build the generators of a free sub-semigroup by conjugating a given hyperbolic element.
Hence the proofs require different techniques.
Moreover, for our purpose, it is important that the cardinality of $W$ grows linearly with the one of $U$. For the optimality of our estimates discussed in the previous paragraph, we require that it is contained in $U^2$. 
The prize that we pay for this is the correction term of the order of the $\ell^{\infty}$-energy of $U$.

As the set $W$ constructed in Theorem  \ref{T: main acylindrical - intro} freely generates a free sub-semigroup, we obtain the following estimate on the growth of $U^r$.

\begin{cor}[see Corollary~\ref{T: acylindrical}]
\label{T: acylindrical - intro}
	Let $G$ be a group acting acylindrically on a hyperbolic length space $X$.
	There exists a constant $C > 0$ such that for every finite $U\subset G$ with $\lambda(U)> C$,  and for all integers $r\geqslant 0$, we have
 $$ |U^r|\geqslant \left(\frac 1{C\lambda(U)} |U| \right)^{[(r+1)/2]} .$$
\end{cor}

As in the previous statement, the constant $C$ actually only depends on the parameters of the action of $G$ on $X$.
Corollary \ref{T: acylindrical - intro} is a variant of \cite[Theorem 1.14]{delzant_product_2020}, where the correction term of the order of $\log|U|$ in this theorem is replaced by a geometric quantity, the $\ell^{\infty}$-energy of $U$. 
Note that the conclusion is void whenever $|U| \leq C \lambda(U)$.
This is typically the case if $U$ generates an elementary subgroup. 
Indeed, since $\lambda(U)$ is assumed to be large, $U$ cannot generate an elliptic subgroup.
Thus $\left< U \right>$ is virtually cyclic and its cardinality is bounded from above linearly in terms of the energy.
This can be compared with Theorem~\ref{T: intro - large powers free burnside} which is not relevant for small subsets $V$.

\subsection{Strategy for Burnside groups}
\label{sec: strategy for burnside groups}
Let us explain the main idea behind the proof of Theorem~\ref{T: intro - large powers free burnside}.
For simplicity we restrict ourselves to the case of free Burnside groups of rank $2$.
Let $n$ be a sufficiently large odd exponent.
Any known strategy to prove the infiniteness of $B_2(n)$ starts in the same way.
One produces a sequence of groups
\begin{equation}
\label{eqn: approx sequence}
	\mathbb F_2 = G_0 \to G_1 \to G_2 \to \dots \to G_i \to G_{i +1} \to \dots
\end{equation}
that converges to $B_2(n)$ where each $G_i$ is a hyperbolic group obtained from $G_{i-1}$ by means of small cancellation.
The approach provided by Delzant and Gromov associates to each group $G_i$ a hyperbolic space $X_i$ on which it acts properly co-compactly.
An important point is that the geometry of $X_i$ is somewhat ``finer'' than the one of the Cayley graph of $G_i$.
In particular, one controls uniformly along the sequence $(G_i, X_i)$, the hyperbolicity constant of $X_i$ as well as the acylindricity parameters of the action of $G_i$, see Proposition~\ref{P: induction step}.
As we stressed before the constant $C$ involved in Theorem~\ref{T: main acylindrical - intro} only depends on those parameters.
Thus it holds, with the same constant $C$, for each group $G_i$ acting on $X_i$.

Consider now a subset $V\subset B_2(n)$ that is not contained in a finite subgroup. 
Our idea is to choose a suitable step $j$ and a pre-image $U_j$ in $G_j$ such that the $\ell^\infty$-energy $\lambda(U_j)$ is greater than $C$ and at the same time bounded from above by a constant $C'$ that does not depend on $j$.
The strategy for choosing $j$ is the following.
The metric spaces $X_i$ defined above come with uniformly contracting maps $X_i \to X_{i+1}$.
Hence if $\tilde V$ stands for a finite pre-image of $V$ in $\mathbb F_2$, then the energy of its image $V_i$ in $G_i$ is a decreasing sequence converging to zero.
Hence there is a smallest index $j$ such that $V$ admits a pre-image $U_{j+1}$ in $G_{j+1}$ whose energy is at most $C$.
Working with the $\ell^\infty$-energy plays now an important role.
Indeed we have a control of the length of \emph{every} element in $U_j$.
This allow us to lift $U_{j+1}$ to a finite subset $U_j \subset G_j$ whose energy is controlled (i.e. bounded above by some $C'$).
It follows from the minimality of $j$ that the energy of $U_j$ is also bounded from below by $C$.
The details of the construction are given in Section~\ref{sec: growth estimates}.
By Theorem~\ref{T: main acylindrical - intro}, we find a ``large'' subset $W \subset U_j^2$ that freely generates a free sub-semigroup.
By large we mean that the cardinality of $W$ is linearly bounded from below by the cardinality of $U_j$ (hence of $V$).

At this point we get an estimate for the cardinality of $W^r$, hence for the one of $U_j^r\subset G_j$, see Corollary~\ref{T: acylindrical - intro}.
However the map $G_j \to G/G^n$ is not one-to-one.
Nevertheless there is a sufficient condition to check whether two elements $g$ and $g'$ in $G_j$ have distinct images in $G/G^n$: roughly speaking, if none of them ``contains a subword'' of the form $u^m$, with $m \geq n/3$, then $g$ and $g'$ have distinct images in $G/G^n$.
This formulation is purposely vague here. 
We refer the reader to Definition~\ref{def: power-free} for a rigorous definition of power-free elements in $G_j$.
In particular, the projection $G_j \to G/G^n$ is injective when restricted to a suitable set of power-free elements.
Hence it suffices to count the number of power-free elements in $W^r$.
This is the purpose of Sections~\ref{sec: periodic and aperiodic words} and \ref{sec: power-free elements}.
The computation is done by induction on $r$ following the strategy of the first author from \cite{coulon_growth_2013}.

Again, we would like to draw the attention of the reader to the fact that in this procedure, we took great care to make sure that all the involved parameters do not depend on $j$.

\subsection{Burnside groups of even exponent} 
\label{sec: discussion-even-exp}
Burnside groups of even exponent have a considerably different algebraic structure.
For instance it turns out that the approximation groups $G_j$ in the sequence (\ref{eqn: approx sequence}) contain elementary subgroups of the form $D_\infty \times F$ where $F$ is a finite subgroup with arbitrary large cardinality that embeds in a product of dihedral groups.
In particular one cannot control acylindricity parameters along the sequence $(G_i)$, which means that our strategy fails here.
It is very plausible that Burnside groups of large even exponents have uniform uniform exponential growth.
Nevertheless we wonder if Theorem~\ref{T: intro - large powers free burnside} still holds for such groups.
 
\subsection{Acknowledgments}
The first author acknowledges support from the \emph{Centre Henri  Le-\\ besgue} ANR-11-LABX-0020-01 and the \emph{Agence Nationale de la Recherche} under Grant \emph{Dagger} ANR-16-CE40-0006-01. The second author was supported in parts by the \emph{Austrian Science Fund (FWF)} project J 4270-N35 and the \emph{Austrian Science Fund (FWF)} project P 35079-N, and thanks the Universit\'e de Rennes 1 for hospitality during his stay in Rennes.  The second author thanks Thomas Delzant for related discussion during his stay in Strasbourg.
We thank the coffeeshop \emph{Bourbon d'Arsel} for welcoming us when the university was closed down during the pandemic, and for serving a wonderful orange cake.
We thank the referees of their careful reading and helpful comments.

\section{Hyperbolic geometry}

We collect some facts on hyperbolic geometry in the sense of Gromov \cite{gromov_hyperbolic_1987}, see also \cite{coornaert_geometry_1990,ghys_hyperboliques_1990}.

\subsection{Hyperbolic spaces} 
\label{sec: Hyperbolic spaces}

Let $X$ be a metric length space. The distance of two points $x$ and $y$ in $X$ is denoted by $|x-y|$, or $|x-y|_X$ if we want to indicate that we measure the distance in $X$. If $A\subset X$ is a set   and $x$ a point, we write $d(x,A)=\inf_{a\in A}|x-a|$ to denote the distance from $x$ to $A$.
Let $A^{+\alpha}=\{x\in X\mid d(x,A) \leqslant \alpha\}$ be the \emph{$\alpha$-neighborhood} of $A$. 
Given $x, y \in X$, we write $[x,y]$ for a geodesic from $x$ to $y$ (provided that such a path exists). 
Recall that there may be multiple geodesics joining two points.
We recall that the Gromov product of $y$ and $z$ at $x$ is defined by 
$$(y,z)_x=\frac{1}{2}\left( |y-x| + |z-x| - |y-z|\right).$$
We will often use the following facts each of which is equivalent to the triangle inequality: for every $x,y,z,t \in X$,
\begin{equation*}
	(x,y)_t \leq (x,z)_t + |y - z|
	\quad \text{and} \quad
	(x,y)_t \leq (x,y)_z + |t-z|.
\end{equation*}
A similar useful inequality is
\begin{equation}
\label{eqn: triangle inequality gromov product}
	(x,y)_t \leq  (x,y)_z + (x,z)_t, \quad \forall x,y,z,t \in X.
\end{equation}
Indeed, after unwrapping the definition of Gromov's products, it boils down to the triangle inequality.

\begin{df}
Let $\delta\geqslant 0$. The space $X$ is $\delta$-hyperbolic if for every $x$, $y$, $z$ and $t\in X$, the \emph{four point inequality} holds, that is
\begin{equation}
\label{eqn: four point inequality}
	(x,z)_t\geqslant \min \left\{ (x,y)_t,(y,z)_t\right\}- \delta.
\end{equation}
\end{df}

If $\delta=0$ and $X$ is geodesic, then $X$ is an $\mathbb{R}$-tree. 
From now on, we assume that $\delta>0$ and that $X$ is a $\delta$-hyperbolic metric length space.
We denote by $\partial X$ the boundary at infinity of $X$. 
Hyperbolicity has the following consequences.
\begin{lem}[{\cite[Lemma~2.2]{coulon_geometry_2014}}]
\label{res: metric inequalities}
	Let $x$, $y$, $z$, $s$ and $t$ be five points of $X$.
	\begin{enumerate}
		\item \label{enu: metric inequalities - thin triangle}
		\begin{math}
			(x,y)_t \leq \max \left\{ |x - t| - (y,z)_x ,  (x,z)_t  \right\} + \delta,
		\end{math}
		\item \label{enu: metric inequalities - two points close to a geodesic}
		\begin{math}
			|s-t| \leq \left| |x-s| - |x-t| \right| + 2\max\left\{ (x,y)_s, (x,y)_t\right\} + 2\delta,
		\end{math}
		\item \label{enu: metric inequalities - comparison tripod} The distance $|s-t|$ is bounded above by 
	\end{enumerate}
	\begin{displaymath}
		\max\left\{| |x-s| - |x-t|| + 2\max \left\{ (x,y)_s,(x,z)_t\right\}, |x-s| + |x-t| - 2 (y,z)_x \right\} + 4\delta.
	\end{displaymath}
\end{lem}

\subsection{Quasi-geodesics}
\label{sec: quasi-geodesics}
A rectifiable path $\gamma:[a,b]\to X$ is a \emph{$(k,\ell)$-quasi-geodesic} if for all $[a',b']\subset [a,b]$ the $$\length(\gamma[a',b'])\leqslant k |\gamma(a')-\gamma(b')|+\ell;$$ and $\gamma$ is a \emph{$L$-local $(k,\ell)$-quasi-geodesic} if any subpath of $\gamma$ whose length is at most $L$ is a $(k,\ell)$-quasi-geodesic. 
The next lemma is used to construct (bi-infinite) quasi-geodesics.

\begin{lem}[Discrete quasi-geodesics {\cite[Lemma~1]{arzhantseva_lower_2006}}] \label{L: discrete quasi-geodesics}
Let $n \geq 3$.
Let $x_1,\dots, x_n$ be $n$ points of $X$.
Assume that for every $i \in \{2, \dots, n-2\}$,
\begin{equation*}	
	(x_{i-1},x_{i+1})_{x_i} + (x_i,x_{i+2})_{x_{i+1}} < |x_i - x_{i+1}| - 3\delta.
\end{equation*}
Then the following holds
\begin{enumerate}
	\item \label{enu: discrete quasi-geodesics - dist}
	$\displaystyle |x_1 - x_n | \geq \sum_{i=1}^{n-1} |x_i - x_{i+1}| - 2\sum_{i=2}^{n-1}(x_{i-1}, x_{i+1})_{x_i} - 2(n-3)\delta$.
	\item \label{enu: discrete quasi-geodesics - gromov}
	$\displaystyle(x_1, x_n)_{x_j} \leq (x_{j-1}, x_{j+1})_{x_j} + 2\delta$, for every $j \in \{ 2, \dots, n-1\}$.
	\item \label{enu: discrete quasi-geodesics - geo}
	If, in addition, $X$ is geodesic, then $[x_1,x_n]$ lies in the $5\delta$-neighborhood of the broken geodesic $\gamma = [x_1,x_2] \cup \dots \cup [x_{n-1},x_n]$, while $\gamma$ is contained in the $r$-neighborhood of $[x_1,x_n]$, where
	\begin{equation*}
		r = \sup_{2 \leq i \leq n-1} (x_{i-1}, x_{i+1})_{x_i} + 14 \delta.
	\end{equation*}
	\qed
\end{enumerate}
\end{lem}
\begin{rem}
	Note that the result still holds if $n = 1$ or $n = 2$.
	Indeed the statement is mostly void, or follows from the definition of Gromov products.
	One just need to replace the error term $2(n-3)\delta$ in (\ref{enu: discrete quasi-geodesics - dist}) by zero.
	Thus in the remainder of the article, we will invoke Lemma~\ref{L: discrete quasi-geodesics} regardless how points are involved.
\end{rem}

We denote by $L_0$ the smallest positive number larger than $500$ such that
for every $\ell \in [0, 10^5 \delta]$, the Hausdorff distance between any two $L_0\delta$-local $(1,\ell)$-quasi-geodesic with the same endpoints is at most $(2\ell + 5\delta)$. See \cite[Corollaries~2.6 and 2.7]{coulon_geometry_2014}.

\subsection{Quasi-convex subsets}
A subset $Y\subset X$ is \emph{$\alpha$-quasi-convex} if for all two points $x,y\in Y$, and for every point $z\in X$, we have $d(z,Y)\leqslant (x,y)_z + \alpha$. For instance, geodesics are $2\delta$-quasi-convex. 

If $Y\subset X$, we denote by $|\ .\ |_Y$ the length metric induced by the restriction of $|\ . \ |_X$ to $Y$. A subset $Y$ that is connected by rectifiable paths is \emph{strongly-quasi-convex} if it is $2\delta$-quasi-convex and if for all $y,y'\in Y$,
$$|y-y'|_X\leqslant |y-y'|_Y \leqslant |y-y'|_X+8\delta.$$

\subsection{Isometries}

Let $G$ be a group that acts by isometries on $X$. 

Let $g\in G$. The \emph{translation length} of $g$ is $$\|  g \|=\inf_{x\in X} |gx-x|.$$
The \emph{stable translation length} of $g$ is $$\|  g \|^{\infty}=\lim_{n\to \infty} \frac{1}{n} |g^nx-x|.$$
Those two quantities are related by the following inequality:
$\|  g \|^{\infty}\leqslant \|  g \| \leqslant \|  g \|^{\infty}+ 16\delta$. 
See \cite[Chapitre 10, Proposition 6.4]{coornaert_geometry_1990}.
The isometry $g$ is hyperbolic if, and only if, its stable translation length is positive, \cite[Chapitre 10, Proposition 6.3]{coornaert_geometry_1990}. 

\begin{df} Let $d > 0$ and $U\subset G$. The set of \emph{$d$-quasi-fixpoints} of $U$ is defined by 
$$\Fix(U,d)=\{x\in X\mid \hbox{for all $u\in U$ } |ux-x|< d\}.$$
The  \emph{axis of $g\in G$} is the set
$A_g= \Fix(g,\|  g \|+8\delta).$
\end{df}

\begin{lem}[\hspace{-0.15mm}{\cite[Lemma 2.9]{coulon_infinite_2018}}] \label{L: convexity fixpoints} \label{L: midpoint} Let $U\subset G$. If $d>7\delta$, then the set of $d$-quasi-fixpoints of $U$ is $10\delta$-quasi-convex. Moreover, assuming that $\Fix(U,d)$ is non-empty
\begin{enumerate}
\item if $x\in X\setminus \Fix(U,d)$, then 
$$\sup_{u\in U} |ux-x| \geqslant 2d(x,\Fix(U,d)) +d -14\delta.$$
\item given $x \in X$ and $L \geq 0$, if $\sup_{u\in U}|ux-x| \leq d + 2L$, then $x \in \Fix(U,d)^{+L+ 7 \delta}.$\qed
\end{enumerate}
 
\end{lem}

\begin{cor}[{\hspace{-0.15mm}\cite[Proposition 2.3.3]{delzant_courbure_2008}}] \label{C: axis} Let $g$ be an isometry of $X$.  Then $A_g$ is $10\delta$-quasi-convex and $g$-invariant. Moreover, for all $x\in X$, 
$$\|  g \|+ 2d(x, A_g) - 6\delta \leqslant |gx-x| \leqslant \|  g \| + 2d(x, A_g) +8\delta.$$ \qed
\end{cor}

\subsection{Acylindricity} We recall the definition of an acylindrical action. 
The action of $G$ on $X$ is \emph{acylindrical} if there exists two functions $N,\kappa \colon \mathbb R_+ \to \mathbb R_+$ such that for every $r \geq 0$, for all points $x$ and $y$ at distance $|x-y|\geqslant \kappa(r)$, there are at most $N(r)$ elements $g\in G$ such that $|x-gx|\leqslant r$ and $|y-gy|\leqslant r$. 

Recall that we assumed $X$ to be $\delta$-hyperbolic, with $\delta > 0$.
In this context, acylindricity satisfies a local-to-global phenomenon: if there exists $N_0,\kappa_0 \in \mathbb R_+$ such that for all points $x$ and $y$ at distance $|x-y|\geqslant \kappa_0$, there are at most $N_0$ elements $g\in G$ such that $|x-gx|\leqslant 100\delta$ and $|y-gy|\leqslant 100\delta$, then the action of $G$ is acylindrical, with the following estimates for the functions $N$ and $\kappa$:
\begin{equation}
\label{eqn: local-global acyl}
	\kappa(r) = \kappa_0 + 4r + 100\delta
	\quad \text{and} \quad
	N(r) = \left( \frac r{5\delta} + 3\right)N_0.
\end{equation}
See \cite[Proposition~5.31]{dahmani_hyperbolically_2017}.
This motivates the next definition.

\begin{df}\label{D: acylindrical} 
Let $N, \kappa \in \mathbb R_+$.
The action of $G$ on $X$ is \emph{$(N,\kappa)$-acylindrical} if for all points $x$ and $y$ at distance $|x-y|\geqslant \kappa$, there are at most $N$ elements $g\in G$ such that $|x-gx|\leqslant 100 \delta$ and $|y-gy|\leqslant 100 \delta$. 
\end{df}

We need the following geometric invariants of the action of $G$ on $X$.
The limit set of $G$ acting on $X$ consists of the accumulation points in the Gromov boundary $\partial X$ of $X$ of the orbit of one (and hence any) point in $X$. By definition, a subgroup $E$ of $G$ is {elementary} if the limit set of $E$ consists of at most two points.

\begin{df} The \emph{injectivity radius} is defined as $$\tau(G,X)=\inf \{\|  g \|^{\infty} \mid g \in G \hbox{ is a hyperbolic isometry}\}.$$
\end{df}

\begin{df} The \emph{acylindricity parameter} is defined as
$$A(G,X)=\sup_{U} \ \diam\left( \Fix(U,2L_0\delta)\right),$$
where $U$ runs over the subsets of $G$ that do not generate an elementary subgroup.
\end{df}

\begin{df} The \emph{$\nu$-invariant} is the smallest natural number $\nu=\nu(G,X)$ such that for every $g\in G$ and every hyperbolic $h\in G$ the following holds: if $g$, $hgh^{-1}$, $\ldots$, $h^{\nu}gh^{-\nu}$ generate an elementary subgroup, then so do $g$ and $h$.
\end{df}

\begin{rem}
\label{rem: convention diam}
	In the above definitions we adopt the following conventions.
	The diameter of the empty set is zero.
	If $G$ does not contain any hyperbolic isometry, then $\tau(G,X) = \infty$.
	If every subgroup of $G$ is elementary, then $A(G,X) = 0$.
\end{rem}

The parameters $A(G,X)$ and $\nu(G,X)$ allow us to state the following version of Margulis' lemma.

\begin{prop}[Proposition~3.5 of \cite{coulon_infinite_2018}]\label{P: margulis 2} Let $U$ be a subset of $G$. 
If $U$ does not generate an elementary subgroup, then, for every $d>0$, we have
$$ \diam\left(\Fix(U,d)\right) \leqslant \left[\nu(G,X) +3\right] d + A(G,X) + 209 \delta.$$ \qed
\end{prop}

If there is no ambiguity we simply write $\tau(G)$, $A(G)$, and $\nu(G)$ for $\tau(G,X)$, $A(G,X)$, and $\nu(G,X)$ respectively.
Sometimes, if the context is clear, we even write $\tau$, $A$, or $\nu$.

If the action of $G$ on $X$ is $(N,\kappa)$-acylindrical, then $\tau\geq \delta/N$, while $A$ and $\nu$ are finite.  
In fact, one could express upper bounds on $A$ and $\nu$ in terms of $N$, $\kappa$, $\delta$, and $L_0$.
See for instance \cite[Section 6]{coulon_partial_2016}. 
However, for our purpose we need a finer control on these invariants.

From now on we assume that $\kappa\geqslant \delta$ and that the action of $G$ on $X$ is $(N,\kappa)$-acylindrical.

\subsection{Loxodromic subgroups}  An elementary subgroup is loxodromic if it contains a hyperbolic element. Equivalently, an elementary subgroup is loxodromic if it has exactly two points in its limit set. 
If $h$ is a hyperbolic isometry, we denote by $E(h)$ the maximal loxodromic subgroup containing $h$. 
Let $E^+(h)$ be the maximal subgroup of $E(h)$ fixing pointwise the limit set of $E(h)$.
It is known that the set $F$ of all elliptic elements of $E^+(h)$ forms a (finite) normal subgroup of $E^+(h)$ and the quotient $E^+(h)/F$ is isomorphic to $\mathbb Z$.
We say that $h$ is \emph{primitive} if its image in $E^+(h)/F$ generates the quotient.

\begin{df}[Invariant cylinder] Let $E$ be a loxodromic subgroup with limit set $\{ \xi,\eta \}$.  The \emph{$E$-invariant cylinder}, denoted by $C_E$, is the $20\delta$-neighborhood of all $L_0 \delta$-local  $(1, \delta)$-quasi-geodesics with endpoints $\xi$ and $\eta$ at infinity. 
\end{df}

\begin{lem}[Invariant cylinder]\label{L: invariant cylinder}
Let $E$ be a loxodromic subgroup. Then
\begin{itemize}
\item $C_E$ is $2\delta$-quasi-convex and invariant under the action of $E$. If, in addition, $X$ is proper and geodesic, then $C_E$ is strongly quasi-convex \cite[Lemma 2.31]{coulon_geometry_2014},

\item if $g\in E$ and $\|  g \|>L_0\delta$, then $A_g\subset C_E$, \cite[Lemma 2.33]{coulon_geometry_2014},

\item  if $g\in E$ is hyperbolic, then $C_E\subset A_g^{+ 52\delta}$. In particular, if $x\in C_E$,  then $|gx-x|\leqslant \|  g \| + 112 \delta $, \cite[Lemma 2.32]{coulon_geometry_2014}.

\end{itemize}     
\end{lem}

\section{Periodic and aperiodic words}
\label{sec: periodic and aperiodic words}
Let $U$ be a finite subset of $G$. 
We denote by $U^*$ the free monoid generated by $U$.
We write $\pi \colon U^* \to G$ for the canonical projection.
In case there is no ambiguity, we make an abuse of notations and still write $w$ for an element in $U^*$ and its image under $\pi$.
We fix a base point $p \in X$.
Recall that the action of $G$ on $X$ is $(N,\kappa)$-acylindrical.

\begin{df}
\label{def: reduced set}
	Let $\alpha > 0$.
	We say that the subset $U$ is \emph{$\alpha$-reduced (at $p$)} if 
	\begin{itemize}
		\item $(u_1^{-1}p, u_2p)_p  \leqslant \alpha$ for every $u_1, u_2 \in U$,
		\item $|up - p| > 2\alpha + 300\delta$ for every $u \in U$.
	\end{itemize}
	The set $U$ is \emph{$\alpha$-strongly reduced (at $p$)} if, in addition, for every distinct $u_1, u_2 \in U$, we have
	\begin{equation*}
		(u_1p,u_2p)_p < \min \left\{ |u_1p - p|, |u_2p - p|\right\} -\alpha - 150\delta.
	\end{equation*}
	We say that $U$ is \emph{reduced at $p$} (respectively \emph{strongly reduced at $p$}) if there exists $\alpha > 0$ such that $U$ is $\alpha$-reduced at $p$ (respectively $\alpha$-strongly reduced at $p$).
\end{df}

In practice, the base point $p$ is fixed once and for all.
Thus we simply say that $U$ is ($\alpha$-)reduced or ($\alpha$-)strongly reduced.

\begin{lem}
\label{res: free-sub-semigroup}
	If $U$ is $\alpha$-strongly reduced, then $U$ freely generates a free sub-semi-group of $G$.
	Moreover $U$ satisfies the \emph{geodesic extension property}, that is if $w, w' \in U^*$ are such that $(p, w'p)_{wp} \leq \alpha +145\delta$, then $w$ is a prefix of $w'$.
\end{lem}

\begin{rem}
	Roughly speaking, the geodesic extension property has the following meaning: if the geodesic $[p, w'p]$ extends $[p, wp]$ as a path in $X$, then $w'$ extends $w$ as a word over $U$.
\end{rem}

\begin{proof}
	We first prove the geodesic extension property. 
	Let $w = u_1\cdots u_m$  and $w' = u'_1\cdots u'_{m'}$ be two words in $U^*$ such that $(p, w'p)_{wp} \leq \alpha +145\delta$. 
	We denote by $r$ the largest integer such that $u_i = u'_i$ for every $i \in \{1, \dots, r-1\}$.
	For simplicity we let 
	\begin{equation*}
		q = u_1\cdots u_{r-1} p = u'_1 \cdots u'_{r-1}p.
	\end{equation*}
	Assume now that contrary to our claim $w$ is not a prefix of $w'$, that is $r - 1 < m$.
	We claim that $(wp, w'p)_q < |u_rp - p| - \alpha - 148\delta$.
	If $r-1 = m'$, then $w'p = q$ and the claim holds.
	Hence we can suppose that  $r-1 < m'$.
	It follows from our choice of $r$ that $u_r \neq u'_r$.
	We let
	\begin{equation*}
		t = u_1\cdots u_rp 
		\quad \text{and} \quad 
		t' = u'_1 \cdots u'_rp.
	\end{equation*}
	Since $U$ is $\alpha$-strongly reduced, we have
	\begin{equation*}
		\left(t,t'\right)_q  = (u_rp, u'_rp)_p < \min \left\{ |u_rp -p|, |u'_rp - p|\right\} - \alpha - 150\delta.
	\end{equation*}
	It follows then from the four point inequality that 
	\begin{equation}	\label{eqn: free-sub-semigroup - 1}
	\begin{split}
		\min\left\{ (t,wp)_q, (wp,w'p)_q, (w'p,t')_q\right\}  & \leq (t,t')_q +2\delta
		\\& < \min \left\{ |u_rp -p|, |u'_rp - p|\right\} - \alpha - 148\delta.
		\end{split}
	\end{equation}
	Applying Lemma~\ref{L: discrete quasi-geodesics}(\ref{enu: discrete quasi-geodesics - gromov}) with the sequence of points 
	\begin{equation*}
		q = u_1\cdots u_{r-1}p, \quad t = u_1\cdots u_rp, \quad u_1\cdots u_{r+1}p, \quad\dots\quad,  wp = u_1\cdots u_mp,
	\end{equation*}
	we get
	\begin{equation*}
		(q,wp)_t \leq (q, u_1\cdots u_{r+1}p)_t + 2 \delta
		= (u_r^{-1}p, u_{r+1}p)_p + 2 \delta
		\leq \alpha + 2\delta.
	\end{equation*}
	(note that the last inequality follows from the fact that $U$ is $\alpha$-reduced).
	Hence
	\begin{equation*}
		(t,wp)_q = |q-t| - (q,wp)_t \geq |u_rp - p| - \alpha - 2\delta.
	\end{equation*}
	Thus the minimum in (\ref{eqn: free-sub-semigroup - 1}) cannot be achieved by $(t,wp)_q$.
	Similarly, it cannot be achieved by $(w'p,t')_q$ either.
	Thus
	\begin{equation*}
		\left(wp,w'p\right)_q <  \min \left\{ |u_rp -p|, |u'_rp - p|\right\} - \alpha - 148\delta \leq |u_rp - p| - \alpha - 148\delta,
	\end{equation*}
	which completes the proof of our claim.

	Using Lemma~\ref{L: discrete quasi-geodesics}(\ref{enu: discrete quasi-geodesics - dist}) with the sequence of points
	\begin{equation*}
		q = u_1\cdots u_{r-1}p, \quad t = u_1\cdots u_rp, \quad u_1\cdots u_{r+1}p, \quad\dots\quad, wp = u_1\cdots u_mp,
	\end{equation*}
	we get
	\begin{equation*}
		|wp -p| \geq \sum_{j = r}^m |u_jp - p| - 2 \sum_{j = r}^{m-1} (u_j^{-1}p,u_{j+1}p)_p - 2\max\{m-r-1,0\}\delta
	\end{equation*}
	Since $U$ is $\alpha$-reduced, we have
	\begin{equation*}
		\sum_{j = r}^m |u_jp - p| > |u_rp - p| + (m-r)(2\alpha + 300\delta),
	\end{equation*}
	while 
	\begin{equation*}
		2 \sum_{j = r}^{m-1} (u_j^{-1}p,u_{j+1}p)_p \leq 2(m-r)\alpha.
	\end{equation*}
	Consequently $|wp - q| \geq |u_rp - p|$.
	Combined with the previous claim, it yields
	\begin{equation*}
		(q, w'p)_{wp} = |wp - q| - (wp,w'p)_q \geq |u_rp - p| - (wp,w'p)_q > \alpha +148\delta.
	\end{equation*}
	Applying again the four point inequality, we get
	\begin{equation}
	\label{eqn: free-sub-semigroup - 2}
		\min \left\{(p,q)_{wp}, (q, w'p)_{wp}\right\}  \leq (p,w'p)_{wp} + \delta \leq \alpha + 146\delta.
	\end{equation}
	It follows from our previous computation that the minimum cannot be achieved by $(q,w'p)_{wp}$.
	We proved previously that $|wp - q| \geq |u_rp - p|$.
	Reasoning as in our first claim, Lemma~\ref{L: discrete quasi-geodesics}(\ref{enu: discrete quasi-geodesics - gromov}) yields $(p,wp)_q \leq \alpha + 2\delta$.
	Since $U$ is $\alpha$-reduced we get
    \begin{equation*}
        		(p,q)_{wp} = |wp-q| - (p,wp)_q \geq |u_rp - p| - \alpha - 2\delta > \alpha+298\delta.
        \end{equation*}
	Hence the minimum in (\ref{eqn: free-sub-semigroup - 2}) cannot be achieved by $(p,q)_{wp}$ either, which is a contradiction.
	Consequently $w$ is a prefix of $w'$.
	
	Let us prove now that $U$ freely generates a free sub-semi-group of $G$.
	Let $w_1, w_2 \in U^*$ whose images in $G$ coincide. 
	In particular $(p, w_1p)_{w_2p} = 0 = (p,w_2p)_{w_1p}$.
	It follows from the geodesic extension property that $w_1$ is a prefix of $w_2$ and conversely. 
	Thus $w_1 = w_2$ as words in $U^*$.
\end{proof}

\subsection{Periodic words} 
From now on, we assume that $U$ is $\alpha$-strongly reduced (in the sense of Definition~\ref{def: reduced set}).  
We let $\lambda = \max_{u\in U}|up-p|$.
We denote by $|w|_U$ the word metric of $w\in U^*$.
Given an  element $w = u_1 \cdots u_m$ in $U^*$, we let 
$$[w]=\{p, u_1p,u_1u_2p,\ldots, wp\}.$$

\begin{df}\label{D: periodic}
Let $m \geq 0$.
Let $E$ be a maximal loxodromic subgroup.
We say that a word $v\in U^*$ is \emph{$m$-periodic with period $E$} if $[v]\subset C_{E}^{+\alpha + 100\delta}$ and $|p-vp| > m \tau(E).$ 
\end{df}

\begin{rem} 
Note that the definition does not require $m$ to be an integer.
Let $E$ be a maximal loxodromic subgroup such that $p$ belongs to the $(\alpha + 100\delta)$-neighborhood of $C_E$.
Let $v \in U^*$ whose image in $G$ is a hyperbolic element of $E$.
Then for every integer $m\geq0$,
the element $v^{m+1}$ is $m$-periodic with period $E$.
The converse is not true; that is, an $m$-periodic word with period $E$ is not necessarily contained in $E$. 
\end{rem}

If $m$ is sufficiently large, then periods are unique in the following sense.

\begin{prop} 
\label{res: unicity period}
There exists $m_0 \geq 0$ which only depends on $\delta$, $A$, $\nu$, $\tau$ and $\alpha$ such that for every $m \geq m_0$ the following holds.
If $v \in U^*$ is $m$-periodic with periods $E_1$ and $E_2$, then $E_1=E_2$. 
\end{prop}

\begin{proof} 
	Let $h_1\in E_1$ realise $\tau(E_1)$, and $h_2\in E_2$ realise $\tau(E_2)$. 
	If $v$ is $m$-periodic with period $E_1$ and $E_2$, then 
	\begin{align*}
		\diam \left( C_{E_1}^{+ \alpha + 100\delta} \cap C_{E_2}^{+\alpha + 100\delta}\right)
		> m \max \{\|  h_1 \|^\infty,\|  h_2 \|^\infty\}.
	\end{align*}
	Recall that $C_{E_i}\subset A_{h_i}^{+52\delta}$, see Lemma \ref{L: invariant cylinder}.
	By \cite[Lemma 2.13]{coulon_geometry_2014} we have
	\begin{equation*}
		\diam \left( C_{E_1}^{+ \alpha + 100\delta} \cap C_{E_2}^{+\alpha + 100\delta}\right)
		\leq \diam \left(A_{h_1}^{+13\delta} \cap A_{h_2}^{+13\delta} \right) + 2\alpha + 308\delta.
	\end{equation*}
	Hence there exists $m_0 \geq 0$ which only depends on $\delta$, $A$, $\nu$, $\tau$ and $\alpha$ such that if $m \geq m_0$, we have
	\begin{equation*}
		\diam \left(A_{h_1}^{+13\delta} \cap A_{h_2}^{+13\delta} \right)
		>  (\nu +2)\max\{\|  h_1 \|,\|  h_2 \|\} +A + 680\delta.
	\end{equation*}
	It follows from \cite[Proposition 3.44]{coulon_partial_2016} that $h_1$ and $h_2$ generates an elementary subgroup, hence $E_1 = E_2$.
\end{proof}

 \begin{rem} 
 \label{R: growth 1} 
For all $w\in U^*$,  we have $\lambda |w|_U \geqslant |wp-p|.$
In particular, if $w$ is an $m$-periodic word with period $E$, then 
$$ |w|_U > m \tau(E)/\lambda.$$

Consider now a general non-empty word $w = u_1\cdots u_r$ in $U^*$.
We claim that $|wp-p|> 2 \alpha + 298\delta|w|_U$.  
Indeed applying Lemma~\ref{L: discrete quasi-geodesics}(\ref{enu: discrete quasi-geodesics - dist}) with the sequence of points
	\begin{equation*}
		p, \quad u_1p, \quad u_1u_2p, \quad\dots\quad, wp = u_1\cdots u_rp,
	\end{equation*}
	we get
	\begin{equation*}
		|wp -p| \geq \sum_{j = 1}^r |u_jp - p| - 2 \sum_{j = 1}^{r-1} (u_j^{-1}p,u_{j+1}p)_p - 2\max\{r-2,0\}\delta
	\end{equation*}
	Since $U$ is $\alpha$-reduced, we have
	\begin{equation*}
		\sum_{j = 1}^r |u_jp - p| > r(2\alpha + 300\delta),
	\end{equation*}
	while 
	\begin{equation*}
		2 \sum_{j = 1}^{r-1} (u_j^{-1}p,u_{j+1}p)_p \leq 2(r-1)\alpha.
	\end{equation*}
Combining the previous inequalities we get the announced estimate.
Consequently if $[w] \subset  C_E^{+\alpha + 100\delta}$ but $w$ is not $m$-periodic with period $E$, then
$$|w|_U < m  \tau(E)/\delta.$$ 
\end{rem}

\begin{prop}
\label{L: construction semi-group power sensitive}
	Let $E$ be a maximal loxodromic subgroup.
	Let $m \geq 0$.
	There are at most two elements in $U^*$ which are $m$-periodic with period $E$, but whose proper prefixes are not $m$-periodic.
\end{prop}

\begin{proof} Let $E$ be a maximal loxodromic subgroup.
Let $\mathcal P_E$ be the set of $m$-periodic words $w \in U^*$ with period $E$.
Assume that $\mathcal P_E$ is non-empty, otherwise the statement is void.
Let $\eta^-$ and $\eta^+$ be the points of $\partial X$ fixed by $E$ and $\gamma \colon \mathbb R \to X$ be an $L_0\delta$-local $(1, \delta)$-quasi-geodesic from $\eta^-$ to $\eta^+$.
For any $w \in \mathcal P_E$, the points $p$ and $wp$ lie in the $(\alpha + 100\delta)$-neighborhood of $C_E$, hence in the $(\alpha + 120\delta)$-neighborhood of $\gamma$.
Without loss of generality, we can assume that $q = \gamma(0)$ is a projection of $p$ on $\gamma$.
We decompose $\mathcal P_E$ in two parts as follows: an element $w \in \mathcal P_E$ belongs to $\mathcal P_E^+$ (respectively $\mathcal P_E^-$) if there is a projection $\gamma(t)$ of $wp$ on $\gamma$ with $t \geq 0$ (respectively $t \leq 0$).
Observe that a priori $\mathcal P_E^-$ and $\mathcal P_E^+$ are not disjoint, but that will not be an issue for the rest of the proof.

We are going to prove that $\mathcal P^+_E \cap U^*$ contains at most one word satisfying the proposition.
Let $w_1$ and $w_2$ be two words in $\mathcal P^+_E \cap U^*$ which are $m$-periodic with period $E$, and whose proper prefixes are not $m$-periodic.
We write $q_1 = \gamma(t_1)$ and $q_2 = \gamma(t_2)$ for the respective projections of $w_1p$ and $w_2p$ on $\gamma$.
Without loss of generality we can assume that $t_1 \leq t_2$.
We are going to prove that $(p, w_2p)_{w_1p} \leq \alpha + 145\delta$.
As a quasi-geodesic, $\gamma$ is $9\delta$-quasi-convex \cite[Corollary~2.7(2)]{coulon_geometry_2014}.
According to Remark~\ref{R: growth 1}, the word $w_2$ is not empty and $|w_2p - p| > 2\alpha + 298\delta$.
Applying the triangle inequality we get $|q_2-q| > 19\delta$.
Recall that $q$ and $q_2$ are respective projections of $p$ and $w_2p$ on the quasi-convex $\gamma$.
Hence 
\begin{equation*}
	|w_2p - p| \geq |w_2p - q_2| + |q_2 - q| + |q - p| - 38\delta,
\end{equation*}
see \cite[Corollary~2.12(2)]{coulon_geometry_2014}.
Since $q_1$ lies on $\gamma$ between $q$ and $q_2$ we also have
\begin{equation*}
	|q_2 - q| = |q_2 - q_1| + |q_1 - q| - 2(q_2,q)_{q_1}
	\geq |q_2 - q_1| + |q_1 - q| - 12\delta,
\end{equation*}
see \cite[Corollary~2.7(1)]{coulon_geometry_2014}.
Combining the previous two inequalities, we get
\begin{align*}
	|w_2p - p| & \geq |w_2p - q_2| + |q_2 - q_1| + |q_1 - q| + |q - p| - 50\delta \\
	& \geq |w_2p -  q_1| + |q_1 - p| - 50\delta
\end{align*}
Thus $(w_2p,p)_{q_1} \leq 25\delta$.
According to the triangle inequality, we get
\begin{equation*}
	(p, w_2p)_{w_1p} \leq |w_1p - q_1| + (w_2p,p)_{q_1} \leq \alpha + 145\delta,
\end{equation*}
which completes the proof of our claim.

Applying the geodesic extension property (see Lemma~\ref{res: free-sub-semigroup}) we get that $w_1$ is a prefix of $w_2$.
As $w_1$ is $m$-periodic, it cannot be a proper prefix, hence $w_1 = w_2$.
Similarly, $\mathcal P^-_E \cap U^*$ has at most one element satisfying the statement.
\end{proof}

\subsection{The growth of aperiodic words}

\begin{df} Let $w\in U^*$ and let $E$ be a maximal loxodromic subgroup. 
We say that the word $w$ \emph{contains an $m$-period of $E$} if $w$ splits as $w=w_0w_1w_2$, where the word $w_1$ is $m$-periodic with period $E$.
If the word $w$ does not  contain any $m$-period, we say that $w$ is \emph{$m$-aperiodic}.
\end{df} 

Observe that containing a period is a property of the word $w \in U^*$ and not of its image $\pi(w)$ in $G$: 
one could find two words $w_1$ and $w_2$, where $w_1$ is $m$-aperiodic while $w_2$ is not, and that have the same image in $G$.
However since $U$ is strongly reduced, it freely generates a free sub-semigroup of $G$.
Hence this pathology does not arise in our context.

We denote by $U^*_m$ the set of $m$-aperiodic words in $U^*$.
Recall that $p$ is a base point of $X$ and the parameter $\lambda$ is defined by 
\begin{equation*}
	\lambda = \max_{u \in U} | up - p|.
\end{equation*}

\begin{ex}\label{E: proper power} If $m \geq \lambda/\tau$, then $U\subseteq U^*_m$. 
Indeed, for all $u\in U$ and loxodromic subgroups $E$, 
$$ |u|_U \leq 1 \leq m\tau/\lambda \leq m \tau(E)/\lambda.$$
So, by Remark~\ref{R: growth 1}, $u$ cannot be $m$-periodic. 
\end{ex}

We denote by $S(r)$ the \emph{sphere} of radius $r$ in $U^*$.
Similarly $B(r) \subset U^*$ stands for the \emph{ball} of radius $r$, that is the subset of elements $w\in U^*$ of word length $|w|_U \leqslant r$. 
We note that $|B(r)|\leqslant |U|^{r+1}$, whenever $|U| \geq 2$.

\begin{prop}\label{P: aperiodic counting} 
Let $U$ be a $\alpha$-strongly reduced subset of $G$, with at least two elements.
There exists $m_1$ which only depends on $\lambda$, $\alpha$, $A$, $\nu$, $\tau$, and $\delta$ with the following property.
For all $m \geq m_1$, and $r>0$, we have
$$|U^*_m \cap B(r+1)| \geqslant \frac{|U|}2 |U^*_m \cap B(r)|.$$
\end{prop}

\begin{proof}
We adapt the counting arguments of \cite{coulon_growth_2013}. 
We firstly fix some notations.
Let $m_0$ be the parameter given by Proposition~\ref{res: unicity period}.
Recall that $m_0$ only depends on $\alpha$, $A$, $\nu$, $\tau$, and $\delta$.
Let $U \subset G$ be an $\alpha$-strongly reduced subset, with at least two elements.
Let $m > m_0 + 5 \lambda/\tau$.
We let 
$$Z=\{ w\in U^* \mid w= w_0u ,\, w_0\in U^*_m ,\, u\in U\}.$$
We denote by $\mathcal{E}$ the set of all maximal loxodromic subgroups in $G$.
For each $E \in \mathcal E$, let $Z_E\subset Z$ be the subset of all $w\in Z$ that split as a product $w=w_1w_2$, where $w_1 \in U^*_m$ and $w_2\in U^*$ is an $m$-periodic word with period $E$.

\begin{lem}\label{L: growth 2}
The set $Z\setminus \bigcup_{E\in \mathcal{E}}Z_E $ is contained in $U^*_m$. 
\end{lem}
\begin{proof}
Let $w\in Z$ contain an $m$-period of a loxodromic subgroup $E \in \mathcal E$. 
By definition of $Z$, we have $w=w_0u$, where $u \in U$ and the prefix $w_0\in U^*$ does not contain any $m$-period.
On the other hand $w$ contains a subword $w_2$ which is an $m$-period with period $E$.
Since $w_2$ cannot be a subword of $w_0$, it is a suffix of $w$.
\end{proof}

Recall that if $W \subset U^*$, then $|W|$ stands for the cardinality of \emph{the image} of $W$ in $G$.
However, since $U$ freely generates a free sub-semi-group (Lemma~\ref{res: free-sub-semigroup}), we can safely identify the elements of $U^*$ with their images in $G$.
It follows from Lemma \ref{L: growth 2}, that for all natural numbers $r$, 
\begin{align}
 \label{I: growth 1}
|U^*_m\cap B(r)| \geqslant |Z\cap B(r)| - \sum_{E\in \mathcal{E}} |Z_E \cap B(r)| .
\end{align}
The next step is to estimate each term in the above inequality.

\begin{lem} \label{I: growth 2}
For all real numbers $r$, 
\begin{align*}
|Z\cap B(r+1)| \geqslant |U| |U^*_m \cap B(r)| .
\end{align*} 
\end{lem}

\begin{proof}
	It is a direct consequence of the fact that $U$ freely generates a free sub-semi-group.
\end{proof}

\begin{lem}\label{L: growth 3} Let $E\in \mathcal{E}$. 
For all real numbers $r$, 
\begin{equation*}
	|Z_E\cap B(r)| \leqslant 2  |\, U^*_m \cap B\left(r-m \tau(E)/\lambda\right)|.
\end{equation*}
\end{lem}
\begin{proof} 
Let $w\in Z_E\cap B(r)$. By definition, $w$ splits as a product $w=w_1w_2$, where $w_1 \in U^*_m$ and $w_2\in U^*$ is $m$-periodic with period $E$. 
By Remark \ref{R: growth 1}, $|w_2|_U> m  \tau(E)/\lambda$, so that $w_1 \in U^*_m \cap B(r-m  \tau(E)/\lambda)$. 

Since $w$ also belongs to $Z$, the prefix consisting of all but the last letter does not contain $m$-periods.
Thus every proper prefix of $w_2$ cannot be $m$-periodic.
It follows from Lemma~\ref{L: construction semi-group power sensitive} that there are at most two possible choices for $w_2$.
Hence the result.
\end{proof}

\begin{lem}\label{I: growth 3} For all real numbers $r$, the following inequality holds:
\begin{align*}
	\sum_{E\in \mathcal{E}} |Z_E\cap B(r)| \leqslant
	2|U|^{m_0\tau/\delta +2}
	\sum_{j\geqslant 1} |\, U^*_m \cap B\left(r -jm\tau/\lambda\right)|\; |U|^{ jm_0\tau/\delta}.
\end{align*}
\end{lem}

\begin{rem}
	Note that the terms in the series on the right hand side are all non-negative.
	Hence if the series diverges, the statement is void.
	Later we will apply this lemma in a setting where the series actually converges.
\end{rem}

\begin{proof}  
Given $j \geq 1$, we define $\mathcal E_j$ as the set of all maximal loxodromic subgroups $E \in \mathcal E$, such that $j\tau \leq \tau(E) < (j+1)\tau$ and $U^*$ contains a word that is $m$-periodic with period $E$.
We split the left-hand sum as follows
\begin{equation*}
	\sum_{E\in \mathcal{E}} |Z_E\cap B(r)|
	= \sum_{j \geq 1} \sum_{E\in \mathcal{E}_j} |Z_E\cap B(r)|
\end{equation*}
Indeed if $U^*$ does not contain a word that is $m$-periodic with period $E$, then the set $Z_E$ is empty.
Observe that for every $E \in \mathcal E_j$ we have by Lemma~\ref{L: growth 3}
\begin{equation*}
	|Z_E\cap B(r)| \leqslant 2 |\, U^*_m \cap B\left(r -jm\tau/\lambda\right)|.
\end{equation*}
Thus it suffices to bound the cardinality of $\mathcal E_j$ for every $j \geq 1$.

Let $j \geq 1$.
For simplicity we let $d_j = (j+1)m_0\tau/\delta + 1$.
We claim that $| \mathcal E_j | \leq |U|^{d_j +1}$.
To that end we are going to build a one-to-one map from $\chi\colon\mathcal E_j  \to B(d_j)$.
Indeed the cardinality of the ball $B(d_j)$ is at most $|U|^{d_j+1}$.
Let $E\in \mathcal{E}_j$.
By definition there exists $w \in U^*$ which is $m$-periodic with period $E$.
Let $w'$ be the shortest prefix of $w$ that is $m_0$-periodic with period $E$.
Note that such prefix always exists since $m\geq m_0$.
By Remark~\ref{R: growth 1}, $w'$ belongs to $B(m_0  \tau(E)/\delta+1)$ hence to $B(d_j)$.
We define $\chi(E)$ to be $w'$.
Observe that there is at most one $E$ such that $w'$ is $m_0$-periodic with period $E$ (Proposition~\ref{res: unicity period}). 
Hence $\chi$ is one-to-one.
This completes the proof of our claim and the lemma. 
\end{proof}

We now complete the proof of Proposition~\ref{P: aperiodic counting}.
Let us define first some auxiliary parameters.
We fix once for all an arbitrary number $\epsilon \in (0,1/2)$.
In addition we let
\begin{equation*}
	\mu = (1-\epsilon) |U|, \,
	\gamma = |U|^{m_0\tau/\delta}, \,
	\xi = 2 |U|^{m_0\tau/\delta + 2}, \,
	\sigma = \frac{ \epsilon}{2(1-\epsilon) \xi},
	\, \text{and }
	M = \left\lfloor \frac{m\tau}\lambda \right\rfloor.
\end{equation*}
Since $|U| \geq 2$, we observe that $\sigma \leq 1/2$.
We claim that there exists $m_1 \geq m_0$ which only depends on $\lambda$, $\alpha$, $A$, $\nu$, $\tau$,  and $\delta$ such that 
\begin{equation*}
	\frac \gamma{\mu^M} \leq \sigma,
\end{equation*}
provided that $m \geq  m_1$.
The computation shows that 
\begin{equation*}
	\ln \left( \frac \gamma{\sigma \mu^M} \right)
	\leq \left(\frac {2 m_0\tau}\delta  + 3 - \frac {m\tau}\lambda\right) \ln |U| - \ln\left(\frac \epsilon {4(1-\epsilon)}\right) - \frac {m\tau}\lambda \ln (1-\epsilon).
\end{equation*}
Recall that $|U| \geq 2$.
Hence if 
\begin{equation*}
	m \geq \frac {2m_0\lambda}\delta  +  \frac{3\lambda}{\tau}
\end{equation*}
then the previous inequality yields
\begin{equation}
\label{eqn: aperiodic counting}
	\ln \left( \frac \gamma{\sigma \mu^M} \right)
	\leq - \frac{m\tau}\lambda \left[ \ln 2 + \ln (1 - \epsilon)\right] +
	\left(\frac {2 m_0\tau}\delta + 3 \right) \ln 2 -  \ln\left(\frac \epsilon {4(1-\epsilon)}\right).
\end{equation}
We can see from there, that there exists $m_1 \geq m_0$ which only depends on $\lambda$, $m_0$, $\tau$, and $\delta$,  such that as soon as $m \geq m_1$ the right hand side of Inequality~(\ref{eqn: aperiodic counting}) is non-positive, which completes the proof of our claim.
Up to increasing the value of $m_1$, we can assume that $M \geq 1$, provided $m \geq m_1$.

Let us now estimate the number of aperiodic words in $U^*$.
From now on we assume that $m \geq m_1$.
For every integer $r$, we let 
\begin{equation*}
	c(r) = | U^*_m \cap B(r)|.
\end{equation*}
We claim that for every integer $r$, we have $c(r) \geq \mu c(r-1)$.
The proof goes by induction on $r$.
In view of Example~\ref{E: proper power}, the inequality holds true for $r = 1$. 
Assume that our claim holds for every $s \leq r$.
In particular for every integer $t \geq 0$, we get $c(r -t) \leq \mu^{-t}c(r)$.
It follows from (\ref{I: growth 1}) that 
\begin{equation*}
	c(r+1) \geq | Z \cap B(r+1)| - \sum_{E \in \mathcal E} |Z_E\cap B(r+1)|.
\end{equation*}
Applying Lemmas~\ref{I: growth 2} and \ref{I: growth 3}, we get
\begin{equation*}
	c(r+1) \geq |U| c(r) - \xi
	\sum_{j\geqslant 1} c(r+1 -jM)\gamma^j.
\end{equation*}
Note that $jM - 1 \geq 0$, for every $j \geq 1$.
Thus applying the induction hypothesis we get
\begin{equation*}
	c(r+1)
	\geq \left( 1 -\frac {\xi\mu} {|U|}\sum_{j \geq 1} \left(\frac{\gamma}{\mu^M}\right)^j\right) |U| c(r).
\end{equation*}
We defined $\mu$ as $\mu = (1 - \epsilon)|U|$, hence it suffices to prove that 
\begin{equation*}
	\frac {\xi\mu} {|U|}\sum_{j \geq 1} \left(\frac{\gamma}{\mu^M}\right)^j \leq \epsilon.
\end{equation*}
Recall that $\gamma/\mu^M \leq \sigma \leq  1/2$.
Hence the series converges. 
Moreover
\begin{equation*}
	\frac {\xi\mu} {|U|}\sum_{j \geq 1} \left(\frac{\gamma}{\mu^M}\right)^j
	\leq \frac {\xi\mu} {|U|} \frac {\sigma}{1 - \sigma}
	\leq \frac {2\xi\mu\sigma} {|U|}
	\leq \epsilon.
\end{equation*}
This completes the proof of our claim for $r+1$.
\end{proof}


\section{Power-free elements}
\label{sec: power-free elements}

Let $G$ be a group that acts $(N,\kappa)$-acylindrically on a $\delta$-hyperbolic geodesic space $X$.
We fix a basepoint $p \in X$.
Recall our convention: the diameter of the empty set is zero, see Remark~\ref{rem: convention diam}.

\begin{df} 
\label{def: power-free}
	Let $m \geq 0$. 
	An element $g\in G$ \emph{contains an $m$-power} if there is a maximal loxodromic subgroup $E$ and a geodesic $[p,gp]$ such that 
	\begin{equation*}
	\diam \left( [p,gp]^{+5\delta} \cap C_E^{+5\delta}\right) > m \tau(E).
	\end{equation*}
	If $g\in G$ does not contain any $m$-power, we say that $g$ is \emph{$m$-power-free}.
\end{df}

Let $U\subset G$ be a finite subset. 
We recall that $\lambda = \max_{u\in U}|up-p|$ and that $U^*$ is the set of all words over the alphabet $U$.
The idea of the next statement is the following.
Take a word $w \in U^*$.
If $w$, seen as an element of $G$, contains a sufficiently large power, then the \emph{word} $w$ already contains a large period.

\begin{prop}
\label{P: proper power vs strong period}
	Let $m \geq (2\lambda + 20\delta)/\tau$.
	Let $U\subset G$ be a finite $\alpha$-reduced subset.
	Let $w\in U^*$. 
	If $w$ contains an $m$-power (as an element of $G$), then $w$ contains an $m'$-period (as a word over $U$), where $m' = m - (2\lambda + 20\delta)/\tau$.
\end{prop}

\begin{proof} 
Let $w=u_1\cdots u_l$. 
As $w$ contains a $m$-power, there is a loxodromic subgroup $E$ and a geodesic $[p,wp]$ such that 
$$\diam\left([p,wp]^{+5\delta} \cap C_E^{+5\delta}\right) > m  \tau(E).$$ 
Let $x_1, x_2$ in $[p,wp]^{+5\delta} \cap C_E^{+5\delta}$ such that $|x_1 - x_2| > m  \tau(E)$.
Let $\gamma_w= [p,u_1p] \cup u_1[p,u_2p] \cup \ldots \cup (u_1\cdots u_{l-1})[p,u_lp]$ be a broken geodesic joining $p$ to $wp$.  
Let $p_1$ and $p_2$ be the respective projections of $x_1$ and $x_2$ on $\gamma_w$.
By Lemma~\ref{L: discrete quasi-geodesics}, the geodesic $[p,wp]$ is contained in the $5\delta$-neighborhood of $\gamma_w$.
Hence $p_1$ and $p_2$ are $15\delta$-close to $C_E$.
Moreover,
$$|p_1 -p_2| \geqslant |x_1-x_2| - 20\delta > m  \tau(E) - 20\delta.$$
Up to permuting $x_1$ and $x_2$ we can assume that $p$, $p_1$, $p_2$ and $wp$ are ordered in this way along $\gamma_w$.
In particular, there is $i \leqslant l-1$ such that $p_1 \in (u_1\cdots u_i) \cdot [ p, u_{i+1} p]$, and $j\leqslant l-1$ such that $p_2 \in (u_1\cdots u_j) \cdot [ p, u_{j+1} p]$. 
Since $p_1$ comes before $p_2$ on $\gamma_w$, we have $i \leq j$.
Note that actually $i < j$.
Indeed if $i=j$,  we would have
\begin{equation*}
	\lambda \geq |u_{i+1}p -p| \geq |p_1-p_2| > m \tau(E) - 20\delta \geq m\tau - 20\delta,
\end{equation*}
which contradicts our assumption.
Let us set $w_0=u_1\cdots u_{i+1}$ and take the word $w_1$ such that  $u_1\cdots u_j = w_0w_1 $.
At this stage $w_1$ could be the empty word.
But we will see that this is not the case.
Indeed
$$|p_1 -p_2| \leqslant  |p_1 - w_0p| + |w_0p-w_0w_1p| + |w_0w_1p-p_2| \leqslant |p-w_1p|+2\lambda.$$ 
Thus, 
$$|p-w_1p|>
 	m  \tau(E) -2 \lambda - 20\delta
 	\geq m'  \tau(E)$$
Applying Lemma~\ref{L: discrete quasi-geodesics} to the subpath $\gamma'$ of $\gamma_w$ bounded by $p_1$ and $p_2$, we get that $\gamma'$ lies in the $(\alpha + 14\delta)$-neighborhood of the geodesic $[p_1, p_2]$.
However $p_1$ and $p_2$ are in the $15\delta$-neighborhood of $C_E$ which is $2\delta$-quasi-convex.
Thus $\gamma'$ is contained in the $(\alpha +31\delta)$-neighborhood of $C_E$.
We conclude that $w_1$ is  $m'$-periodic with period $w_0^{-1}Ew_0$.  
\end{proof}

\section{Energy and quasi-center}

Let $G$ be a group acting by isometries on a $\delta$-hyperbolic length space $X$.  
Recall that we assume for simplicity that $\delta > 0$.
In next sections, we denote by $S(x,r)$ the sphere in $X$ of radius $r$ centered at $x$.
(This should not be confused with the spheres in $U^*$ used in the previous section.)
Let $U\subset G$ be a finite subset.
In order to apply the counting results from Section~\ref{sec: periodic and aperiodic words}, we explain in this section and the followings how to build a strongly reduced subset of $U^2$.
To that end we define the notion of energy of $U$.

\begin{df} 
The \emph{$\ell^{\infty}$-energy $\lambda(U,x)$ of $U$ at $x$} is defined by 
$\lambda(U,x) = \max_{u\in U} |ux-x|$.
The \emph{$\ell^{\infty}$-energy} of $U$ is given by  $$\lambda(U)=\inf_{x\in X} \lambda(U,x).$$
A point $q\in X$ is \emph{almost-minimising the $ \ell^{\infty}$-energy} if $\lambda(U,q)\leqslant \lambda(U)+\delta$. 
\end{df}

\label{sec: quasi-centre}
Let $x \in X$ and $A, B \subseteq X$.
Define $U_x(A,B)$ to be the set of elements $u \in U$ satisfying the following conditions
\begin{itemize}
	\item $|x - ux| \geq  4 \cdot 10^3\delta$,
	\item there exists $a \in A \cap S(x,10^3\delta)$, such that $(x,ux)_{a}\leqslant  \delta$,
	\item there exists $b \in B \cap S(x,10^3\delta)$, such that $(u^{-1}x,x)_b\leqslant \delta$.
\end{itemize}
We write $U_x(A)=U_x(A,A)$, and, if there is no ambiguity, $U(A,B)=U_x(A,B)$ for short.

\begin{df}[Quasi-centre]
A point $x\in X$ is a \emph{quasi-centre for $U$} if, for all $y\in S(x,10^3\delta)$, we have
$$\left|U_x\left(y^{+100\delta}\right)\right|\leqslant \frac{3}{4} |U|.$$
\end{df}

\begin{prop}\label{P: quasi-centre}
Let $q$ be a point that almost-minimises the $\ell^{\infty}$-energy of $U$.
There exists a quasi-centre $p$ for $U$ such that $| p - q | \leqslant \lambda (U)$. 
\end{prop}
\begin{rem} The existence of a quasi-centre is already known by  \cite{delzant_product_2020}.
The authors prove there that any point almost-minimising the $\ell^1$-energy is a quasi-centre.
However such a point could be very far from any point almost-minimising the $\ell^\infty$-energy.
\end{rem}

\begin{proof} 
We describe a recursive procedure to find a quasi-centre $p$. 
The idea is to construct a quasi-geodesic from $q$ to a quasi-centre $p$.
Let $x_0=q$ and suppose that  $x_0$, $\ldots$, $x_{i-1}$, $x_i\in X$ are already defined.
If $x_i$ is a quasi-centre for $U$, we let $p=x_i$ and stop the induction. 
Otherwise, there is a point $x_{i+1} \in S(x_i,10^3\delta)$ such that $|U_{x_i}(x_{i+1}^{+ 100 \delta})| > \frac{3}{4} |U|$.

Our idea is to apply Lemma \ref{L: discrete quasi-geodesics} to the sequence of points $x_0,$ $x_1,$ $\dots,$ $x_i,$ $x_{i+1},$ $ux_{i+1}$, $ ux_i,$ $\dots,$ $ux_1,$ $ux_0$ for some $u \in U_{x_i}(x_{i+1}^{+ 100 \delta})$. Like this we can write the distance from $x_0$ to $ux_0$ as a function of the index $i$.
We will observe that this function diverges to infinity, which forces the procedure to stop.
To do this, we collect the following observations. By construction, we have:

\begin{lem}\label{L: stable points 1} 
For all $u \in U_{x_{i-1}}(x_{i}^{+100\delta})$, the following holds
\begin{enumerate}
	\item $(x_{i-1},ux_{i-1})_{x_i}\leqslant 101\delta$ and $(x_{i-1},u{x_{i-1}})_{ux_i}\leqslant 101\delta$
	\item $(x_{i-1},ux_i)_{x_i} \leq 102\delta$ and $(x_i, ux_{i-1})_{ux_i} \leq 102\delta$.
\end{enumerate}
\end{lem}

\begin{rem}
Roughly speaking, this lemma tells us that  $x_{i-1},$ $x_i,$ $ux_i$ and $ux_{i-1}$ are aligned in the order of their listing along the neigbourhood of the geodesic $[x_{i-1},ux_{i-1}]$.
\end{rem}

\begin{proof}
	The first point is just a reformulation of the definition of the set $U_{x_{i-1}}(x_{i}^{+100\delta})$.
	Let $u \in U_{x_{i-1}}(x_{i}^{+100\delta})$.
	By Lemma~\ref{res: metric inequalities}~(\ref{enu: metric inequalities - thin triangle}) we have
	\begin{equation}
	\label{eqn: stable points 1}
		(x_{i-1},ux_i)_{x_i}\leq \max \left\{|x_{i-1}- x_i| - (ux_i, ux_{i-1})_{x_{i-1}}, (x_{i-1},ux_{i-1})_{x_i}\right\} + \delta.
	\end{equation}
	According to the triangle inequality we have
	\begin{equation*}
		(ux_i, ux_{i-1})_{x_{i-1}} \geq |ux_{i-1} - x_{i-1}| - |x_{i-1} - x_i|.
	\end{equation*}
	However, by construction $|ux_{i-1} - x_{i-1}| > 2|x_{i-1} - x_i| + 2\delta$.
	Hence the maximum in (\ref{eqn: stable points 1}) has to be achieved by $(x_{i-1},ux_{i-1})_{x_i}$.
	The same argument works for $(x_i, ux_{i-1})_{ux_i}$.
\end{proof}

\begin{lem}\label{L: stable points 2}
If $x_i$ is not a quasi-centre for $U$, then $(x_{i-1},x_{i+1})_{x_i}\leqslant 103\delta$.
\end{lem}

\begin{proof}
We note that $ |U_{x_{i-1}}(x_{i}^{+100\delta}) \cap U_{x_{i}}(x_{i+1}^{+100\delta})|>|U|/2.$
Let us fix an element $u$ in this intersection. 
By Lemma \ref{L: stable points 1}, $(x_{i-1},ux_i)_{x_i} \leq 102\delta$ and $(x_{i},ux_{i})_{x_{i+1}}\leq 101\delta$.
According to the four point inequality we have
\begin{equation*}
	102 \delta 
	\geq (x_{i-1},ux_i)_{x_i}
	\geq \min \left\{ (x_{i-1},x_{i+1})_{x_i}, (x_{i+1},ux_i)_{x_i}\right\} - \delta.
\end{equation*}
Observe that 
\begin{equation*}
	|x_i - x_{i+1}| 
	= (x_{i+1}, ux_i)_{x_i}  +(x_i,ux_i)_{x_{i+1}} 
	\leq (x_{i+1}, ux_i)_{x_i} + 101\delta.
\end{equation*}
Since $|x_i - x_{i+1}| = 10^3\delta$, the minimum cannot be achieved by $(x_{i+1},ux_i)_{x_i}$, whence the result.
\end{proof}

\begin{lem}\label{L: stable points 3} If $x_i$ is not a quasi-centre, then, for all $u\in U_{x_i}(x_{i+1}^{+100 \delta})$,  
$$|ux_0 -x_0 | \geq |{x_{i+1}}-ux_{i+1}|  + 10^3(i+1) \delta .$$
\end{lem}

\begin{proof}
Let $u$ be in $U_{x_i}(x_{i+1}^{+100 \delta})$. 
By Lemma \ref{L: stable points 1},
 \begin{equation*}
 x_i,ux_{i+1})_{x_{i+1}} \leq 102\delta \text{ and } (x_{i+1},ux_i)_{ux_{i+1}} \leq 102\delta.
 \end{equation*}
On the other hand, by Lemma \ref{L: stable points 2}, we have  
\begin{equation*}
(x_{j-1},x_{j+1})_{x_{j}} \leqslant 103\delta \text{ and } (ux_{j-1},ux_{j+1})_{ux_{j}}\leqslant 103\delta\text{, for all $0< j \leqslant i$}.
\end{equation*}
The claim follows from Lemma \ref{L: discrete quasi-geodesics} applied to the sequence of points $x_0,$ $x_1,$ $\dots,$ $x_i,$ $x_{i+1},$ $ux_{i+1}$, $ ux_i,$ $\dots,$ $ux_1,$ $ux_0$.
\end{proof}

Suppose that $x_i$ is not a quasi-centre. 
Fix $u\in U_{x_i}(x_{i+1}^{+100 \delta})$.
By construction we have $|{x_{i+1}}-ux_{i+1}|\geq 4\cdot 10^3\delta$.
Recall that $x_0 = q$  almost-minimises the energy.
By Lemma \ref{L: stable points 3}, we get
$$\lambda(U) \geq  10^3(i+5) \delta. $$
This means that the induction used to build the sequence $(x_i)$ stops after finitely many steps.
Moreover, when the process stops we have $x_i = p$ and $\lambda(U) \geq  10^3(i+5) \delta$.
For every $j \leq i-1$ we have $| x_j - x_{j+1}| \leq 10^3\delta$, thus $|p - q| \leq \lambda(U)$.
\end{proof}

\section{Sets of diffuse energy} 
In this section we assume that the action of $G$ on $X$ is $(N, \kappa)$-acylindrical, with $\kappa > 50\cdot 10^3\delta$.
Let $U \subset G$ be a finite subset.
Let $p$ be a quasi-centre of $U$.
In this section we assume that $U$ is of  \emph{diffuse energy} (at $p$) that is for at least $\frac{99}{100}$ of the elements of $U\subset G$, we have $ |up-p|> 2\kappa$.

  \subsection{Reduction lemma}
 
 We first prove the following variant of the reduction lemmas in \cite{delzant_product_2020}.   
 
 \begin{prop}[Reduction]\label{L: reduction} 
There is $v\in U$, and $U_1\subset U$ of cardinality $|U_1|\geqslant \frac{1}{100} |U|$ such that for all $u_1\in U_1$, 
 \begin{itemize}
 \item $(u_1^{-1}p,vp)_p\leqslant 10^3 \delta$ and $(v^{-1}p,u_1p)_p\leqslant 10^3 \delta$, and 
 \item $2\kappa \leqslant |u_1p-p|\leqslant |vp-p|$. 
 \end{itemize}
 \end{prop}
 
 If $U$ has only one element, then $U_1=U=\{v\}$, the isometry $v$ is loxodromic
and $p$ is in a small neighbourhood of the axis of $v$.

\begin{rem} In the case of trees, Proposition \ref{L: reduction} follows directly from \cite[Lemma 6.4]{delzant_product_2020}, and the proof of this lemma is due to Button \cite{button_explicit_2013}. The situation is different in the case of hyperbolic spaces. Indeed, in contrast to the reduction lemmas in \cite[Section 6.1]{delzant_product_2020}, the cardinality of $U_1$ in Proposition  \ref{L:  reduction} does not depend the cardinality of balls in $X$, as in \cite[Lemma 6.3]{delzant_product_2020}, and the estimates on the Gromov products do not depend on the logarithm of the cardinality of $U$, as in \cite[Lemma 6.8]{delzant_product_2020}. 
\end{rem}

\begin{proof}
For simplicity we let $\eta = 1/100$.
Let  $U'=\{u\in U\mid |up-p|\geqslant 2\kappa\}$. 
As the energy of $U$ is diffuse at $p$, we have $|U'|\geq (1-\eta) |U|$. 
Let us fix $u_0\in U'$ such that $|u_0p-p|$ is maximal.
We claim the following result.

\begin{lem}\label{L: main technical lemma reduction} 
At least one of the following holds:
\begin{enumerate}
\item there is $U_1\subset U'$ of cardinality $|U_1|>\eta |U|$ such that for all $u_1\in U_1$,
 $$ (u_1^{-1}p,u_0p)_p\leqslant 10^3 \delta \hbox{ and } (u_0^{-1}p,u_1p)_p\leqslant 10^3 \delta;$$
\item there are  $U_1$, $U_2 \subset U'$ of cardinalities $|U_1|>\eta |U|$  and $|U_2|>\eta |U|$ such that for all $u_1\in U_1$, $u_2\in U_2$, 
 $$ (u_1^{-1}p,u_2p)_p\leqslant 10^3 \delta \hbox{ and } (u_2^{-1}p,u_1p)_p\leqslant 10^3 \delta.$$
\end{enumerate}

\end{lem}

We postpone for the moment the proof of this lemma and complete first the demonstration of Proposition~\ref{L: reduction}.
In case (1) of Lemma \ref{L: main technical lemma reduction}, we set $v=u_0$. In case (2) of Lemma \ref{L: main technical lemma reduction}, we may assume, up to exchanging the roles of $U_1$ and $U_2$, that there is $v\in U_2$ such that for all $u_1\in U_1$, $|u_1p-p|\leqslant |vp-p|$. This yields Proposition~\ref{L: reduction}. 
 \end{proof}

\begin{proof}[Proof of Lemma \ref{L: main technical lemma reduction}] 
	We write $S=S(p,10^3 \delta)$ for short.
	For simplicity, in this proof we write $$U'(A,B) = U' \cap U_p(A,B).$$ See Section~\ref{sec: quasi-centre} for the definition or $U_p(A,B)$.
	
	The definition of hyperbolicity implies the following useful lemma.

	\begin{lem}[Lemma 6.1 of \cite{delzant_product_2020}]
	\label{L: reduction 1} 
		Let $y_1,z_1,y_2,z_2\in S$. 
		If $|z_1-y_2|>6\delta$, then for every $u_1 \in U(y_1, z_1)$ and $u_2 \in U(y_2,z_2)$, we have $ (u_1^{-1}p,u_2p)_p\leqslant 10^3 \delta$. \qed
	\end{lem}
	
	By construction $|u_0p - p| \geq 4\cdot 10^3\delta$.
	Thus there exists $y_0$ and $z_0 \in S$ such that $(u_0p,p)_{y_0}\leqslant \delta$ and $(p,u_0^ {-1}p)_{z_0}\leq \delta$, so that $u_0 \in U'(y_0, z_0)$.
	Assume first that $|U'(S\setminus z_0^{+6\delta},S\setminus y_0^ {+6\delta})|>\eta|U|$.
	Then we let $U_1 = U'(S\setminus z_0^{+6\delta},S\setminus y_0^ {+6\delta})$.
	Using Lemma~\ref{L: reduction 1} we conclude that (1) holds.
	
	Observe that the complement in $U'$ of the previous set is the union of $U'(z_0^{+6\delta}, S)$ and $U'(S,y_0^ {+6\delta})$.
	Recall that $|U'| > (1 -\eta)|U|$.
	Thus we can now assume that 
	\begin{equation}
	\label{eqn: main technical lemma reduction}
	 \left| U'(z_0^{+6\delta}, S) \cup U'(S,y_0^ {+6\delta}) \right| > (1-2\eta)|U|.
	\end{equation}
	Let us now assume that $|U'(z_0^{+6\delta}, S\setminus z_0^{+12\delta}) > \eta |U|$.
	In this case we let $U_1 = U_2 = U'(z_0^{+6\delta}, S\setminus z_0^{+12\delta})$.
	Using Lemma~\ref{L: reduction 1} we conclude that (2) holds.
	The same argument works if $|U'(S\setminus y_0^ {+12\delta},y_0^ {+6\delta})| > \eta|U|$.
	Suppose now that the cardinality of $U'(z_0^{+6\delta}, S\setminus z_0^{+12\delta})$ and $U'(S\setminus y_0^ {+12\delta},y_0^ {+6\delta})$ are both bounded above by $\eta |U|$.
	It follows from (\ref{eqn: main technical lemma reduction}) that 
	\begin{equation}
	\label{eqn: main technical lemma reduction bis}
	 \left| U_1 \cup U_2 \right| > (1-4\eta)|U|, \quad
	 \text{where}\ U_1 = U'(z_0^{+6\delta}, z_0^{+12\delta})\ \text{and}\ U_2 = U'(y_0^ {+12\delta},y_0^ {+6\delta}).
	\end{equation}
	Since $p$ is a quasi-centre, the cardinality of both $U_1$ and $U_2$ is bounded above by $3|U|/4$.
	It follows from (\ref{eqn: main technical lemma reduction bis}) that each of them contains at least $(1/4 - 4\eta)|U|$ elements.
	Observe also that $|y_0 - z_0| > 30\delta$.
	Indeed otherwise both $U_1$ and $U_2$ are contained in $U(y_0^{+100\delta})$.
	Hence (\ref{eqn: main technical lemma reduction bis}) contradicts the fact that $p$ is a quasi-centre.
	Applying Lemma~\ref{L: reduction 1} we conclude that $U_1$ and $U_2$ satisfy (2).
\end{proof}

\subsection{Construction of free sub-semi-groups}

 We recall that $\lambda(U)$ denotes the $\ell^\infty$-energy of the finite subset $U \subset G$. 
By Proposition \ref{P: quasi-centre}, we can assume that the quasi-centre $p$, which we fixed at the beginning of this section, is at distance at most $\lambda (U)$ from a point almost-minimising the $\ell^\infty$-energy of $U$.
We still assume that the energy of $U$ is diffuse (at $p$).
We treat $p$ as the base point of $X$.

\begin{rem}
	According to the triangle inequality, we have $|up - p| \leq 3\lambda(U) + \delta$, for every $u \in U$.
	Since the energy of $U$ is diffuse at $p$, there is an element $u \in U$ that moves $p$ by a large distance.
	As a consequence $\lambda(U) \geq \delta$, and thus $|up - p| \leq 4\lambda(U)$, for every $u \in U$.
	This estimates are far from being optimal, but sharp enough for our purpose.
\end{rem}

\begin{prop}
\label{P: construction semi-group}
	There exists $v \in U$ and a subset $W \subset Uv$ such that $W$ is $1002\delta$-strongly reduced and 
	\begin{equation*}
	    	|W| \geq \frac 1{10^6N}\frac \delta{\lambda(U)} |U|.
	\end{equation*}
\end{prop}

\begin{proof}
For simplicity we let $\alpha = 1002\delta$.
We fix $U_1$ and $v$ given by Proposition~\ref{L: reduction}. 
We set $T=U_1v$.

\begin{lem}
\label{res: basics on T}
	For every $t, t' \in T$, we have $(t^{-1}p,t'p)_p\leq \alpha$ and $|tp - p| > 2\alpha + 300\delta$.
\end{lem}

\begin{proof}
	We write $t = uv$ and $t' = u'v$ with $u,u' \in U_1$.
	Applying twice the four point inequality (\ref{eqn: four point inequality}) we have
	\begin{equation}
	\label{eqn: basics on T}
		\min\left\{ (v^{-1}p, t^{-1}p)_p, (t^{-1}p, t'p)_p, (t'p, u'p)_p \right\} \leq (v^{-1}p, u'p)_p + 2\delta \leq \alpha
	\end{equation}
	Observe that 
	\begin{equation*}
		(v^{-1}p, t^{-1}p)_p 
		= |p - v^{-1}p| - (p, t^{-1}p)_{v^{-1}p}
		= |p - vp| - (vp, u^{-1}p)_p
		\geq 2\kappa - 1000\delta > \alpha.
	\end{equation*}
	Similarly we prove that $(t'p, u'p)_p > \alpha$.
	Hence the minimum in (\ref{eqn: basics on T}) is achieved by $(t^{-1}p, t'p)_p$ which proves the first point.
	By definition of Gromov products we have
	\begin{equation*}
		|tp - p| = |up - p| + |vp - p| - 2 (u^{-1}p,vp)_p
		\geq 4\kappa - 2000\delta. \qedhere
	\end{equation*}
\end{proof}

For every $w \in T$, we set 
\begin{equation*}
	A_w = \left\{ t \in T \mid |p - tp| \leq |p -wp|\ \text{and}\ (wp,p)_{tp} \leq \alpha + 150\delta \right\}.
\end{equation*}
Note that $w \in A_w$.

In order to define $W$, we construct by induction an increasing sequence $(W_i)$ of subsets of $T$.
We first let $W_0 = \emptyset$.
Assume that now that $W_i$ has been defined for some integer $i \geq 0$.
If the set 
\begin{equation*}
	T \setminus \bigcup_{w \in W_i} A_w
\end{equation*}
is empty, then the process stops and we let $W = W_i$ (note that this will ineluctably happen as $T$ is finite).
Otherwise, we choose an element $w_{i+1}$ in this set for which $|p - w_{i+1}p|$ is maximal and let $W_{i+1} = W_i \cup \{w_{i+1}\}$.

\begin{lem}\label{L: construction of sub-semi-group 4}
	The set $W$ is $\alpha$-strongly reduced.
\end{lem}

\begin{proof}
	By Lemma~\ref{res: basics on T}, the set $T$ (hence $W$) is $\alpha$-reduced. 
	It suffices to prove that for every distinct $w,w' \in W$ we have
	\begin{equation*}
		(wp,w'p)_p \leq \min \{ |wp - p|, |w'p-p|\} - \alpha - 150\delta.
	\end{equation*}
	Using the notation above, we write, $w_1, w_2, \dots, w_n$ for the elements $W$ in the order they have been constructed.
	Let $i,j \in \{1, \dots, n\}$ such that $|p - w_jp| \leq |p - w_ip|$.
	If $i < j$, then $w_j$ does not belong to $A_{w_i}$, thus 
	\begin{align*}
		(w_ip, w_jp)_p 
		& = |w_jp-p| - (w_ip, p)_{w_jp} < |w_jp-p| - \alpha - 150\delta \\
		&  \leq \min \{ |w_ip - p|, |w_jp-p|\} - \alpha - 150\delta.
	\end{align*}
	Assume now that $j < i$.
	Note that the sequence $\{|p - w_kp|\}$ is non-increasing, hence $|p - w_jp| = |p -w_ip|$.
	Since $w_i$ does not belong to $A_{w_j}$, thus 
	\begin{align*}
		(w_ip, w_jp)_p 
		& = |w_ip-p| - (w_jp, p)_{w_ip}
		< |w_ip-p| - \alpha - 150\delta \\
		& \leq \min \{ |w_ip - p|, |w_jp-p|\} - \alpha - 150\delta. \qedhere
	\end{align*}
\end{proof}

\begin{lem}\label{L: construction of sub-semi-group 3}
	For every $w \in T$, we have 
	\begin{equation*}
		|A_w| \leq \frac{2065N}{\delta} \lambda(U).
	\end{equation*}
\end{lem}

\begin{proof}
	Let $w \in T$.
	The proof goes in two steps.
	First we give an upper bound for subsets of sparse elements in $A_w$.
	Let $m \geq 0$ be an integer.
	Let $t_0 = u_0v$, $t_1 = u_1v$, \dots, $t_m = u_mv$ be $m$ pairwise distinct elements in $A_w$.
	We assume in addition that $|u_ip - u_jp| > 6\cdot 10^3\delta$, for every distinct $i,j \in \{0, \dots,m\}$.
	Let $\gamma \colon [a,b] \to X$ be a $(1, \delta)$-quasi-geodesic from $p$ to $wp$.
	We are going to give an upper bound for $m$.
	To that end we claim that the points $u_0p, \dots, u_mp$ lie close to $\gamma$.
	Since the points $u_ip$ are sparse, this will roughly say that $m \lesssim |wp-p| / \max\{ |u_ip - u_jp|\}$.
	More precisely, the argument goes as follows.
	For every $i \in \{0, \dots, m\}$, we write $p_i$ for a projection of $u_ip$ onto $\gamma$.
	Up to reindexing the elements we can suppose that the points $p, p_0, p_1, \dots, p_m, wp$ are aligned in this order along $\gamma$.
	
	Since $t_i$ belongs to $A_w$, we have
	\begin{equation*}
		(wp,p)_{t_ip}  \leq \alpha + 150\delta \leq 1152\delta.
	\end{equation*}
	On the other hand, we know by construction of $U_1$ and $v$ that $(p,t_ip)_{u_ip} = (u_i^{-1}p,vp)_p$ is at most $10^3\delta$, see Proposition~\ref{L: reduction}.
	Hence the triangle inequality yields, see (\ref{eqn: triangle inequality gromov product})
	\begin{equation*}
		(p,wp)_{u_ip} \leq  (wp,p)_{t_ip} + (p,t_ip)_{u_ip} \leq 2152\delta.
	\end{equation*}
	Since $\gamma$ is $(1, \delta)$-quasi-geodesic, it is $9\delta$-quasi-convex, see \cite[Corollary~2.7(2)]{coulon_geometry_2014}.
	It follows  that $|u_ip - p_i| = d(u_ip,\gamma)$ is at most $2161\delta$.
	According to the triangle inequality we get
	\begin{equation*}
		|p_i - p_j| > 1678\delta,\quad \forall i \neq j
	\end{equation*}
	Observe now that 
	\begin{equation*}
		1678m\delta \leq \sum_{i = 0}^{m-1} |p_i - p_{i+1}|
		\leq {\rm Length}(\gamma) \leq |p - wp| + \delta
	\end{equation*}
	Recall that $w$ is a two letter word in $U$, while $\lambda(U)$ is very large compare to $\delta$.
	Hence $1678m\delta \leq 9\lambda(U)$.
	To simply the rest of the computations, we will use the following generous estimate
	\begin{equation*}
		m \leq \frac{ \lambda(U)}\delta.
	\end{equation*}

	We now start the second step of the proof.
	Using acylindricity we reduce the counting of elements in $A_w$ to the case of a sparse subset.
	Any element $t \in A_w$ can be written $t = u_tv$ with $u_t \in U_1$.
	Consider now $t,t' \in A_w$. 
	
	We claim that $|u_tvp - u_{t'}vp| \leq |u_tp -u_{t'}p| +4306\delta$. 
	Indeed, by definition of $A_w$, we have $(wp,p)_{u_tvp} \leq  \alpha + 150\delta$ and $(wp,p)_{u_{t'}vp} \leq  \alpha + 150\delta$.
	By Lemma~\ref{res: metric inequalities}(\ref{enu: metric inequalities - two points close to a geodesic}) we have
	\begin{equation}
	\label{eqn: construction of sub-semi-group 3}
	\begin{split}
	|u_tvp - u_{t'}vp|
	& \leq ||p - u_tvp| - |p-u_{t'}vp|| + 2 \max\left\{ (wp,p)_{u_tvp}, (wp,p)_{u_{t'}vp}\right\} + 2\delta \\
	& \leq ||p - u_tvp| - |p-u_{t'}vp|| + 2306\delta.
	\end{split}
	\end{equation}
	Note that 
	\begin{equation*}
		|p-u_tp| +|vp-p| - 2\cdot 10^3\delta \leq |p - u_tvp|\leq |p-u_tp| +|vp-p|. 
	\end{equation*}
	Indeed the second inequality is just the triangle inequality, while the first one is equivalent to the following known fact $(u_t^{-1}p,vp)_p\leqslant 10^3\delta$.
	Similarly we have 
	\begin{equation*}
		|p-u_{t'}p| +|vp-p| - 2\cdot 10^3\delta \leq |p - u_{t'}vp|\leq |p-u_{t'}p| +|vp-p|. 
	\end{equation*}
	The difference of the previous two inequalities yields 
	\begin{equation*}
		||p - u_tvp| - |p-u_{t'}vp|| \leq | |p - u_tp| - |p - u_{t'}p|| + 2\cdot 10^3\delta.
	\end{equation*}
	Plugin this inequality in (\ref{eqn: construction of sub-semi-group 3}) we obtain
\begin{equation*}
		|u_tvp - u_{t'}vp|  \leq \left | |p - u_tp| - |p - u_{t'}p|\right| + 4306 \delta.
	\end{equation*}
Finally, by the triangle inequality 
$
		\left | |p - u_tp| - |p - u_{t'}p|\right| \leq  |u_tp-u_{t'}p| .
$
	This implies the claim. 
	
	We can now take advantage of acylindricity.
	Recall that $|vp - p| \geq 2\kappa$, with $\kappa > 50 \cdot 10^3\delta$.
	In particular, 
	\begin{equation*}
		|vp - p| \geq \kappa + 41324\delta.
	\end{equation*}
	We let $M = 2065N$.
	According to acylindricity -- see (\ref{eqn: local-global acyl}) applied with $r = 10306\delta$ -- the set
	\begin{equation*}
		F = \left\{ g \in G \mid |gp - p| \leq 6000\delta\ \text{and} \ |gvp - vp| \leq 10306\delta \right\}
	\end{equation*}
	contains at most $M$ elements.
	It follows that for every $t \in A_w$, there are at most $M$ elements $t' \in A_w$ such that $|u_tp - u_{t'}p| \leq 6\cdot 10^3\delta$.
	Indeed, if $|u_tp - u_{t'}p| \leq 6\cdot 10^3\delta$, our previous claim implies that $u_t^{-1}u_{t'}$ belongs to $F$.
	
	So we can extract a subset $B \subset A_w$ containing $m  \geq |A_w| / M$ elements such that for every distinct $t,t' \in B$ we have $|u_tp - u_{t'}p| > 6\cdot 10^3\delta$.
	It follows from the previous discussion that $m \leq \lambda(U) / \delta$.
	Consequently,
	\begin{equation*}
		|A_w| \leq \frac{2065N}{\delta} \lambda(U). \qedhere
	\end{equation*}
\end{proof}

\begin{lem}\label{L: construction of sub-semi-group 6}
	The cardinality of $W$ is bounded from below as follows:
	 $$|W|\geqslant \frac 1{2065N}\frac \delta{\lambda(U)} |T|.$$
\end{lem}

\begin{proof}
	Recall that $w \in A_w$ for every $w \in T$.
	Thus, by construction, the collection of sets $\{A_w\}_{w \in W}$ covers $T$.
	We have seen in Lemma~\ref{L: construction of sub-semi-group 3} that the cardinality of each of them is at most $2065N\lambda(U)/\delta$.
	Hence the result.
\end{proof}

The previous lemma completes the proof of Proposition~\ref{P: construction semi-group}.
\end{proof}

\section{Sets of concentrated energy}
We still assume here that the action of $G$ on $X$ is $(N, \kappa)$-acylindrical, with $\kappa > 50\cdot 10^3\delta$.
Let $U\subset G$ be a finite subset and $p \in X$ a base point.
In this section we also assume that $U$ has \emph{concentrated energy} (at $p$) that is, there exists $U_1 \subset U$ with $|U_1| \geq |U|/100$ such that $|up-p|\leqslant 2\kappa$, for all $u\in U_1$.
The goal of the section is to prove the following statement.

\begin{prop}\label{P: main concentrated}
Let $M =2\kappa N/ \delta$.
If $\lambda(U,p)> 100\kappa$, then one of the following holds:
\begin{enumerate}
	\item either $ |U| \leq 100M$;
	\item or there exists $v \in U$ and a subset $W \subset Uv$ such that $W$ is $25\kappa$-strongly reduced and $$|W| \geq |U|/100M - 1.$$ 
\end{enumerate}
\end{prop}

\begin{proof}
We assume that $|U| > 100M$, so that $|U_1| > M$.
The proof follows the exact same ideas as Lemmas~5.2 and 5.3 of \cite{delzant_product_2020}. 
Since the energy $\lambda(U,p)$ at $p$ is larger than $100\kappa$, there exists $v \in U$ satisfying $|vp - p| > 100\kappa$.
For every $u \in U_1$, we let 
\begin{equation*}
	B_u = \left\{ u' \in U_1 \mid (uvp, u'vp)_p \geq 23\kappa - \delta\right\}.
\end{equation*}
Note that by the triangle inequality, $|uvp - p| > |vp - p| - |up - p| \geq 98\kappa$, for every $u \in U_1$.
Hence $u \in B_u$.

Let us fix first an element $u \in U_1$.
We claim that the cardinality of $B_u$ is at most $M$.
Recall that $X$ is a a length space, hence there is a point $m$ in $X$ such that $|p - m | = 21\kappa - \delta$ and $(p,vp)_m \leq \delta$.
Let $u' \in B_u$.
The element $u'u^{-1}$ moves the point $up$ by at most $4\kappa$.
We now show that $u'u^{-1}$ moves $um$ by at most $4\kappa + 8\delta$.
By Lemma~\ref{res: metric inequalities}(\ref{enu: metric inequalities - thin triangle}) we have
\begin{equation}
	\begin{split}
	(p, uvp)_{um} 
	& \leq \max\left\{|uvp - um| - (p,up)_{uvp}, (up,uvp)_{um}\right\} + \delta \\
	& \leq \max\left\{|vp - m| - (p,up)_{uvp}, (p,vp)_{m}\right\} + \delta
	\end{split}
	\label{eqn: main concentrated - eqn1}
\end{equation}
On the one hand, we have
\begin{equation*}
	|vp - m| = |p - vp| - |p - m| + 2 (p,vp)_m \leq |p - vp| - 21 \kappa + 3\delta.
\end{equation*}
On the other hand, the triangle inequality yields
\begin{equation*}
	(p,up)_{uvp} \geq |up - uvp| - |p - up| \geq |p - vp| - 2\kappa.
\end{equation*}
If we plug in the last two inequalities in (\ref{eqn: main concentrated - eqn1}) we get $(p, uvp)_{um} \leq 2 \delta$.
Now observe that
\begin{equation*}
	\left||p-um| - |p - m| \right| = \left||p-um| - |up - um| \right| \leq |p- up| \leq 2\kappa.
\end{equation*}
Similarly $(p, u'vp)_{u'm} \leq 2 \delta$ and $\left||p-u'm| - |p - m| \right| \leq 2\kappa$.
In particular both $|p - um|$ and $|p - u'm|$ are at most $(uvp,u'vp)_p$.
By Lemma~\ref{res: metric inequalities}(\ref{enu: metric inequalities - comparison tripod}) we have
\begin{equation*}
	|um - u'm| \leq \max \left\{ \left||p-um| - |p - u'm| \right| + 4 \delta, 0 \right\}+ 4\delta 
	\leq 4\kappa + 8\delta,
\end{equation*}
which corresponds to our announcement.

Note that the point $up$ and $um$, which are ``hardly'' moved by $u'u^{-1}$, are far away.
More precisely
\begin{equation*}
		|up- um| = |p - m|= 21\kappa - \delta.
\end{equation*}
Recall that $M = 2\kappa N /\delta$.
Using acylindricity -- see (\ref{eqn: local-global acyl}) with $r = 4\kappa + 8\delta$ -- we get that $B_u$ contains at most $M$ elements, which completes the proof of our claim.

Recall that $u \in B_u$, for every $u \in U_1$.
We now fix a maximal subset $U_2 \subset U_1$ such that for every $u \in U_1$, any two distinct $u_1, u_2 \in U_2$ never belong to the same subset $B_u$.
The cardinality of $U_2$ is at least $|U_2| \geq |U_1|/M$.
Indeed by maximality of $U_2$, the $U_1$ is covered by the collection $(B_u)_{u \in U_2}$.

We claim that there is at most one element $u \in U_2$ such that $(v^{-1}p, uvp)_p > 23\kappa$.
Assume on the contrary that it is not the case.
We can find two distinct element $u,u' \in U_2$ such that 
\begin{equation*}
	(uvp,u'vp)_p \geq \min \left\{ (v^{-1}p, uvp)_p, (v^{-1}p, u'vp)_p\right\} - \delta > 23\kappa - \delta.
\end{equation*}
Thus $u'$ belongs to $B_u$ which contradicts the definition of $U_2$.
Recall that $|U_1| > M$, hence $U_2$ contains at least $2$ elements.
We define then $U_3$ from $U_2$ by removing if necessary the element $u \in U_2$ such that $(v^{-1}p, uvp)_p > 23\kappa$.
Note that 
\begin{equation*}
	|U_3| \geq \frac{|U_1|}M - 1 \geq \frac{|U|}{100M} - 1.
\end{equation*}

We now let $W = U_3v$.
We are going to prove that $W$ is $25\kappa$-strongly reduced.
Note first that 
\begin{equation*}
	|wp - p| \geq |vp - p| - 2\kappa > 98 \kappa > 50\kappa + 300\delta
\end{equation*}
for every $w \in W$.
Let $w = uv$ and $w' = u'v$ be two elements in $W$.
It follows from the triangle inequality that 
\begin{equation*}
	(w^{-1}p, w'p) _p \leq (v^{-1}p, w'p)_p + |up - p| \leq  (v^{-1}p, w'p)_p + 2\kappa.
\end{equation*}
By construction of $U_3$, no element $w' \in W$ has a large Gromov product with $v^{-1}$.
Hence $(w^{-1}p, w'p) _p \leq 25\kappa$.
Thus the set $W$ is $25\kappa$-reduced.
By choice of $U_2$ we also have $(wp,w'p)_p < 23\kappa - \delta$ for every distinct $w,w' \in W$.
Recall that 
\begin{equation*}
	\min\left\{ |wp-p|, |w'p-p| \right\} \geq |vp - p| - 2\kappa > 98\kappa.
\end{equation*}
Consequently $W$ is $25\kappa$-strongly reduced.
\end{proof}

\section{Growth in groups acting on hyperbolic spaces}\label{S: basepoint}

As a warmup for the study of Burnside groups we first prove the following statement.

\begin{thm}\label{T: main acylindrical}
Let $\delta >0$, $\kappa\geqslant 50\cdot 10^3\delta$, and $N>0$.
Assume that the group $G$ acts $(N,\kappa)$-acylindrically on a $\delta$-hyperbolic length space.
For  every finite $U\subset G$ such that $\lambda(U)> 100\kappa$, one of the following holds.
\begin{enumerate}
	\item $\displaystyle |U| \leq  400\kappa N/\delta$.
	\item There exists $v \in U$ and a subset $W \subset Uv$ such that $W$ is $\alpha$-strongly reduced, with $\alpha \leq 25\kappa$, and  
	\begin{equation*}
		|W| \geq \frac 1{10^6N}\frac \delta{\lambda(U)} |U|.
	\end{equation*} 
\end{enumerate}

\end{thm}

\begin{proof}[Proof of Theorem \ref{T: main acylindrical}]
Let $U\subset G$ be a finite subset such that $\lambda(U)>100\kappa$. 
 
\textbf{Choice of the base-point.} Let $q$ be a point almost-minimising the $\ell^{\infty}$-energy of $U$. We now fix the base-point $p$ to be a quasi-centre for $U$. By Proposition \ref{P: quasi-centre}, we can assume that $|p-q|\leqslant \lambda(U)$. 

{\textbf{Case 1: diffuse energy.}} Let us first assume that $U$ is of diffuse energy at $p$. That is, there is a subset $U'\subset U$ such that $|U'|\geqslant 99|U|/100$ and such that for all $u'\in U'$ we have $|u'p-p|>2\kappa$. 
Then, by Proposition \ref{P: construction semi-group}, there exists $v \in U$ and a subset $W \subset Uv$ such that $W$ is $\alpha$-strongly reduced (with $\alpha = 1002\delta$) and whose cardinality satisfies 
\begin{equation*}
	|W| \geq \frac 1{10^6N}\frac \delta{\lambda(U)} |U|.
\end{equation*}

{\textbf{Case 2: concentrated energy.}}  Otherwise $U$ is of concentrated energy at $p$. Indeed, there is a subset $U'\subset U$ of cardinality $|U'|\geqslant |U|/100$ such that $|u'p-p|\leqslant 2\kappa$, for all $u'\in U'$.
Recall that $\lambda(U) > 100\kappa$.
Assume that $|U| > 400\kappa N /\delta$.
By Proposition \ref{P: main concentrated}, there exists $v \in U$ and a subset $W \subset Uv$ such that $W$ is $\alpha$-strongly reduced (with $\alpha = 25\kappa$) and whose cardinality satisfies
\begin{equation*}
	|W| \geq \frac 1{200N}\frac \delta{\kappa} |U|-1 \geq \frac 1{10^6N}\frac \delta{\lambda(U)} |U|.
\end{equation*}
This completes the proof of Theorem \ref{T: main acylindrical}. \end{proof}

\begin{cor}\label{T: acylindrical}
	Let $\delta >0$, $\kappa\geqslant 50\cdot 10^3\delta$, and $N>0$.
Assume that the group $G$ acts $(N,\kappa)$-acylindrically on a $\delta$-hyperbolic length space.
	For  every finite $U\subset G$ such that $\lambda(U)> 100\kappa$ and for all integers $r\geqslant 0$, we have
 $$ |U^r|\geqslant \left( \frac 1{10^6N}\frac \delta{\lambda(U)} |U| \right)^{[(r+1)/2]} .$$
\end{cor}

\begin{proof}
	Without loss of generality we can assume that $|U| > 400\kappa N /\delta$.
	Indeed otherwise the base of the exponential function on the right hand side of the stated inequality is less than one, hence the statement is void.
	According to Theorem~\ref{T: main acylindrical}, there exists $v \in U$ and a subset $W \subset Uv$ such that $W$ is $\alpha$-strongly reduced and  
	\begin{equation*}
		|W| \geq \frac 1{10^6N}\frac \delta{\lambda(U)} |U|.
	\end{equation*} 
	Let $s \geq 0$ be an integer.
	On the one hand, $(Uv)^s$ is contained in $U^{2s}$, hence $|U^{2s}| \geq |(Uv)^s|$.
	On the other hand $(Uv)^sU$ is contained in $U^{2s+1}$.
	Right multiplication by $v$ induces a bijection from $G$ to itself.
	Hence
	\begin{equation*}
		|U^{2s+1}| \geq |(Uv)^sU| = |(Uv)^{s+1}|.
	\end{equation*}
	Recall that $W$ is contained in $Uv$ and freely generates a free-sub-semigroup of $G$ by Lemma~\ref{res: free-sub-semigroup}.
	It follows that for every integer $r \geq 0$,
	\begin{equation*}
		|U^r| \geq |(Uv)^{[(r+1)/2]}| \geq |W|^{[(r+1)/2]} \geq \left( \frac 1{10^6N}\frac \delta{\lambda(U)} |U| \right)^{[(r+1)/2]}.\qedhere
	\end{equation*}
\end{proof}

We now combine Theorem \ref{T: main acylindrical} with our estimates on the growth of aperiodic words, see Proposition \ref{P: aperiodic counting}. If we use Proposition \ref{P: proper power vs strong period} to compare the notion of aperiodic words and power-free elements we obtain the following useful growth estimate.

\begin{cor}\label{T: main acylindrical power free}
Let $\delta >0$, $\kappa\geqslant 50\cdot 10^3\delta$, $N>0$ and $\lambda_0 \geq 0$.
There exists a parameter $m_2 > 0$ with the following properties.
Assume that the group $G$ acts $(N,\kappa)$-acylindrically on a $\delta$-hyperbolic geodesic space.
Let $U\subset G$ such that $100\kappa <\lambda(U) \leq \lambda_0$.
One of the following holds.
\begin{enumerate}
	\item $\displaystyle |U| \leq  \max\{4\kappa N/\delta, 4\cdot 10^6 N \lambda(U)/\delta\}$.
	\item There is $v \in U$ with the following property.
	For every $r > 0$ and $m \geq m_2$, denote by $K(m,r)$ the set of all $m$-power-free elements in $(Uv)^r$.  
	Then, 
$$|K(m,r)| \geqslant \left( \frac{1}{4\cdot 10^6 N}\frac{\delta}{\lambda(U)} |U|\right)^{r}. $$ 
\end{enumerate}

\end{cor}

\begin{proof} Let  $U\subset G$ be a finite subset such that $\lambda(U)>100 \kappa$. 
Without loss of generality we can assume that $|U| > \max\{4\kappa N/\delta, 4\cdot 10^6 N \lambda(U)/\delta\}$.
By Theorem \ref{T: main acylindrical} there exists $v \in U$ and a subset $W \subset Uv$ such that $W$ is $\alpha$-strongly reduced with $\alpha \leq 25\kappa$ and  
	\begin{equation*}
		|W| \geq \frac 1{10^6N}\frac \delta{\lambda(U)} |U|.
	\end{equation*}
It follows from our choice that $|W| \geq 4$ and $\lambda(W) \leq 2 \lambda(U)$.

Before moving on, let us recall some notations from Section~\ref{sec: periodic and aperiodic words}.
For every integer $m$, the set $W^*_m$ stands for the collection of $m$-aperiodic words in $W^*$.
In addition $S(r)$ and $B(r)$ are respectively the sphere and the ball of radius $r$ in $W^*$ (for the word metric with respect to $W$).

In view of Proposition~\ref{P: aperiodic counting}, there exists $m_1 > 0$, which only depends on $\delta$, $N$, $\kappa$ and $\lambda_0$ such that for every $m \geq m_1$, for every $r \geq 0$, we have,
\begin{equation}
	|W^*_m \cap B(r+1)| \geqslant \frac{|W|}2 |W^*_m \cap B(r)|. \label{eq: growth aperiodic 1} 
\end{equation}
Let us now focus on the cardinality of \emph{spheres}.
As $W$ is $\alpha$-strongly reduced, it generates a free sub-semi-group (Lemma \ref{res: free-sub-semigroup}). Thus 
$$
|W_m^* \cap S(r+1)| = |W^*_m\cap B(r+1)| -|W_m^* \cap B(r)|.
$$
If we combine this inequality with \eqref{eq: growth aperiodic 1} and the fact that $|W|/4\geqslant 1$, we obtain that 
\begin{align*}
	|W^*_m \cap S(r+1)| 
	\geqslant \frac{|W|}2 |W^*_m \cap B(r)|  -|W_m^* \cap B(r)| 
	&\geqslant \left(\frac{|W|}{2} -1\right) |W^*_m \cap B(r)| \\
	& \geqslant \frac{|W|}4 |W^*_m \cap B(r)| \\
	& \geqslant \frac{|W|}4 |W^*_m \cap S(r)| .
\end{align*}
By an inductive argument, we obtain that, for all $r\geqslant 0$,    
\begin{equation*}
|W^*_m \cap S(r)| \geqslant  \left(\frac{|W|}{4}\right)^r.
\end{equation*}
Now let $m_2=m_1+(2\lambda(S) + 20\delta)/\tau$ and let $m\geqslant m_2$. 
Then, by Proposition \ref{P: proper power vs strong period}, every element in $W^*_{m'}$ (seen as an element of $G$) is $m$-power-free, where $m'=m-(2\lambda(S) + 20\delta)/\tau$ is larger than $m_1$. 
Thus, 
\begin{equation*}
|K(m,r)|\geqslant |W^*_{m'} \cap S(r)|  \geqslant \left(\frac{|W|}{4}\right)^{r}.
\end{equation*}
 This completes the proof. 
\end{proof}

\section{Small cancellation groups}
In this section we recall the necessary background on small cancellation theory with a special attention on acylindricity, see Proposition~\ref{P: small cancellation acylindricity}.
The presentation follows \cite{coulon_geometry_2014} in content and notations.

\subsection{Cones}

Let $Y$ be a metric length space and let $\rho>0$. 
The \emph{cone of radius $\rho$ over $Y$} is the set 
$$Z(Y)= Y\times [0,\rho]/\sim,$$
where $\sim$ is the equivalence relation which identifies all the points of the form $(y,0)$ for $y \in Y$.
If $x\in Z (Y)$, we write $x=(y,r)$ to say that $(y,r)$ represents $x$. We let $v=(y,0)$ be the \emph{apex} of the cone.

If $y$, $y'$ are in $Y$, we let $\theta(y,y')=\min\left\{ \pi,\, |y-y'|/\sinh\rho\right\}$ be their angle at $v$. There is a metric on $Z(Y)$ that is characterised as follows, see \cite[Chapter~I.5]{bridson_metric_1999}. 
Let $x=(y,r)$ and $x'=(y',r')$ in $Z(Y)$. Then 
$$ \cosh |x-x'| = \cosh r \cosh r' -\sinh r \sinh r' \cos \theta (y,y').$$ 
It turns out that $Z(Y)$ is a hyperbolic space \cite[Proposition~4.6]{coulon_geometry_2014}.

We let $\iota: Y \to Z(Y)$ be the embedding defined as $\iota(y)=(y,\rho)$. The metric distortion of $\iota$ is  controlled by a function $\mu: \mathbb{R}_+ \to  [0,2\rho]$ that is characterised as follows: for every $ t\in \mathbb{R}_+$, 
$$ \cosh \mu(t) = \cosh^2 \rho - \sinh^2 \rho \cos\left( \min\left\{\pi, \frac t{\sinh \rho}\right\}\right).$$ 
For all $y,\, y' \in Y$, we have 
\begin{equation}
\label{L: distortion}  |\iota(y)-\iota(y')|_{Z(Y)}=\mu\left(|y-y'|_Y\right).
\end{equation}
Let us mention some properties of $\mu$ for later use.

\begin{prop}[Proposition 4.4 of \cite{coulon_geometry_2014}]\label{P: mu} The map $\mu$ is continuous, concave, non-decreasing. Moreover, if $\mu(t)< 2\rho$, then $t \leqslant \pi \sinh (\mu(t)/2)$. \qed
\end{prop}

Let $H$ be a group that acts by isometries on $Y$. Then $H$ acts by isometries on $Z(Y)$ by $hx=(hy,r)$. We note that $H$ fixes the apex of the cone. 

\subsection{The cone off space}
From now, we assume that $X$ is a proper, geodesic, $\delta$-hyperbolic space, where $\delta >0$.
We fix a parameter $\rho > 0$, whose value will be made precise later.
In addition, we consider a group $G$ that acts properly co-compactly by isometries on $X$.
We assume that this action is $(N,\kappa)$-acylindrical.

We let $\mathcal Q$ be a collection of  pairs $(H, Y) $ such that $Y$ is closed strongly-quasi-convex 
 in $X$ and $H$ is a subgroup of $\stab Y$ acting co-compactly on $Y$. 
Suppose that $\mathcal Q$ is closed under the action of $G$ given by the rule $g(H,Y)=(gHg^{-1}, gY)$.
In addition we assume that $\mathcal Q/G$ is finite.
Furthermore, we let 
\begin{equation*}
	\Delta(\mathcal Q) = \sup \left\{ \diam \left( Y_1^{+5\delta} \cap Y_2^{+5\delta} \right) \mid (H_1, Y_1) \neq (H_2,Y_2) \in \mathcal Q\right\}
\end{equation*}
and
\begin{equation*}
	\mathcal{T}(\mathcal Q)=\inf \{\|  h \|\mid h\in H \setminus \{1\},\, (H,Y) \in Q\}.
\end{equation*}
Observe that if $\Delta(\mathcal Q)$ is finite, then $H$ is normal in $\stab Y$, for every $(H,Y) \in \mathcal Q$.

Let $(H,Y) \in \mathcal Q$. We denote by $|\cdot|_Y$ the length metric on $Y$ induced by the restriction of $|\cdot|$ to $Y$.
As $Y$ is strongly quasi-convex, for all $y,y' \in Y$, 
$$|y-y'|_X\leqslant |y-y'|_Y\leqslant |y-y'|_X+ 8\delta.$$  
We write $Z(Y)$ for the cone of radius $\rho$ over the metric space $(Y, |\cdot|_Y)$.

We let the \emph{cone-off space} $\cX= \cX (Y,\rho)$  be the space obtained by gluing, for each pair $(H,Y)\in \mathcal Q$, the cone $Z(Y)$ on $Y$ along the natural embedding $\iota: Y \to Z(Y)$. 
We let $\mathcal V$ denote the set of apices of $\cX$.  
We endow $\cX$ with the largest  metric $|\cdot|_{\cX}$ such that the map $X\to \cX$  and the maps $Z(Y)\to \cX$ are $1$-Lipschitz,  see \cite[Section 5.1]{coulon_geometry_2014}. It has the following properties. 

\begin{lem}[Lemma 5.7 of \cite{coulon_geometry_2014}]\label{L: metric cone off 2}
Let $(H,Y) \in \mathcal Q$.
Let $x \in Z(Y)$ and $x' \in \cX$.
Let $d(x,Y)$ be the distance from $x$ to $\iota(Y)$ computed in $Z(Y)$.
If $|x-x'|_{\cX} < d(x, Y)$, then $x' \in Z(Y)$ and $|x-x'|_{\cX} = |x-x'|_{Z(Y)}$. 
\end{lem}

We recall that $\mu$ is the map that controls the distortion of the embedding $\iota$ of $Y$ in its cone, see (\ref{L: distortion}). 
 It also controls the distortion of the map $X\to \cX$. 
\begin{lem}[Lemma 5.8 of \cite{coulon_geometry_2014}]\label{L: metric cone off} 
For all $x,\, x' \in X$, we have $\mu(|x-x'|_X)\leqslant |x-x'|_{\cX}\leqslant |x-x'|_X$. \qed
\end{lem}

The action of $G$ on $X$ then extends to an action by isometries on $\cX$: given any $g \in G$, a point $x= (y,r)$ in $Z(Y)$ is sent to the point $gx=(gy,r)$ in $Z(gY)$. 
We denote by $K$ the normal subgroup generated by the subgroups $H$ such that $(H,Y) \in \mathcal Q$. 

\subsection{The quotient space} \label{sec: quotient space}

We let $\q X= \cX/K$ and $\qG = G/K$. 
 We denote by $\zeta$ the projection of $\cX$ onto $\q X$ and write $\q x$ for $\zeta(x)$ for short. Furthermore, we denote by $\q {\mathcal V}$ the image in $\q X$ of the apices $\mathcal V$. 
We consider $\q X$ as a metric space equipped with the quotient metric, that is for every $x,x' \in \cX$ 
$$ |\q x- \q x'|_{\q X} = \inf_{h\in K} |hx-x'|_{\cX}.$$
We note that the action of $G$ on $\cX$ induces an action by isometries of $\q G$ on $\q X$.
The following theorem summarises Proposition~3.15 and Theorem~6.11 of \cite{coulon_geometry_2014}.

\begin{thm}[Small Cancellation Theorem \cite{coulon_geometry_2014}]\label{P: hyperbolicity quotient} 
There are distances $\delta_0$, $\delta_1$, $\Delta_0$ and $\rho_0$ (that do not depend on $X$ or $\mathcal Q$) such that, 
if  $\delta \leq \delta_0$, $\rho > \rho_0$, $\Delta(\mathcal Q) \leq \Delta_0$, and $\mathcal T(\mathcal Q) > 4\pi \sinh\rho$, then the following holds:
\begin{enumerate}
\item $\q X$ is a proper geodesic $\delta_1$-hyperbolic space on which $\bar G$ acts properly co-compactly.
\item Let $r\in (0,\rho/20]$. If for all $v\in \mathcal V$, the distance $|x-v|\geqslant 2r$ then the projection $\zeta \colon \cX \to \qX$ induces an isometry from $B(x,r)$ onto $B(\q x,r)$.
\item Let $(H,Y) \in \mathcal Q$. 
If $v \in \mathcal V$ stands for the apex of the cone $Z(Y)$, then the projection from $G$ onto $\q G$ induces an isomorphism from $\stab Y/H$ onto $\stab{\q v}$.  \qed
\end{enumerate}
\end{thm}

Let us now fix $\delta_0$, $\delta_1$, $\Delta_0$ and $\rho_0$ as in Theorem \ref{P: hyperbolicity quotient}. We assume that $\delta \leq \delta_0$, $\Delta(\mathcal Q) \leq \Delta_0$,  $\mathcal T(\mathcal Q) > 4\pi \sinh\rho$, and $\rho > \rho_0$, so that  $\q X$ is  $\q \delta$-hyperbolic, with $\q \delta \leq \delta_1$.

We use point (2) of Theorem \ref{P: hyperbolicity quotient} to compare the local geometry of $\cX $ and $\q X$. To compare the global geometry, we use the following proposition.

\begin{prop}[Proposition 3.21 of \cite{coulon_geometry_2014}]\label{L: lifts of quasi-convex} Let $\q Z \subset \q X$ be $10 \bar\delta$-quasi-convex and $d\geqslant 10 \bar\delta$.
If, for all $\q v \in \q {\mathcal V}$, we have $\q Z \cap B(\q v, \rho/5 +d+ 1210 \q \delta) =\emptyset$, then there is a pre-image $Z \subset \cX$ such that the projection $\zeta$ induces an isometry from $Z$ onto $\q Z$. 

In addition, if $\q S\subset \q G$ such that $\q S \, \q Z \subseteq {\q Z}^{+d}$, then there is a pre-image $S\subset G$ such that for every $g \in S$, $z,z' \in Z$
, we have $|\q g\, \q z -\q z '| =|g z-z'|_{\cX}$. 
 \qed
\end{prop}

\subsection{Group action on \texorpdfstring{$\qX$}{X}}

We collect some properties of the action of $\q G$.

\begin{lem}[Lemma 6.8 of \cite{coulon_geometry_2014}]\label{L: rotation quotient space} If $v\in \mathcal V$ and $g\in \q G \setminus \stab{\q v}$, then for every $\q x\in \qX$ we have
\begin{equation*}
|\q g \,\q x -\q x|\geqslant 2(\rho - |\q x -\q v|). 
\end{equation*}  \qed
\end{lem}
 
In combination with assertion (2) of Theorem \ref{P: hyperbolicity quotient}, the previous lemma implies that local properties of the action are often inherited from the  action of $G$ on the cone-off space. For example, if $\q F$ is an elliptic subgroup of $\q G$, then either $\q F \subseteq \stab{\q v}$ for some $v\in \mathcal V$ or it is the image of an elliptic subgroup of $G$, see \cite[Proposition~6.12]{coulon_geometry_2014}.

There is a lower bound on the injectivity radius of the action on $\q X$, and an upper bound on the acylindricity parameter.

\begin{prop}[Proposition 6.13 of \cite{coulon_geometry_2014}]\label{P: injectivity radius quotient}
Let $\ell=\inf \{\|  g \|^{\infty}\mid g\not\in \stab Y,\, (H,Y)\in \mathcal Q\}$. Then 
$$\tau(\q G, \q X)\geqslant \min\left\{\frac{\rho\ell}{4\pi \sinh \rho }, \q \delta\right\}.$$ \qed
\end{prop}

We recall that $L_0$ is the number fixed in Section~\ref{sec: quasi-geodesics} using stability of quasi-geodesics.

\begin{prop}[Corollary 6.15 of \cite{coulon_geometry_2014}]\label{P: diameter thin part quotient}
Assume that all elementary subgroups of $G$ are cyclic infinite or finite with odd order.
If $\stab Y$ is elementary for every $(H,Y) \in \mathcal Q$, then 
$A(\q G, \q X)\leqslant A(G,X) + 5 \pi \sinh (2L_0\q  \delta).$ 
 \qed
\end{prop}

Note that the proposition actually does not require that finite subgroups of $G$ have odd order.
This assumption in \cite[Propositions 6.15]{coulon_geometry_2014} was mainly made to simplify the overall exposition in this paper.
The error of the order of $\pi \sinh (2L_0 \q\delta)$ in the above estimates is reminiscent of the distortion of the embedding of $X$ into $\cX$, measured by the map $\mu$, see Proposition~\ref{P: mu}.

\subsection{Acylindricity}

Let us assume that all elementary subgroups of $G$ are cyclic (finite or infinite).   
In particular, it follows that $\nu(G,X) = 1$, see for instance \cite[Lemma~2.40]{coulon_geometry_2014}.
Moreover, we assume that for every pair $(H,Y)\in \mathcal Q$, there is a primitive hyperbolic element $h\in G$ and a number $n$ such that $H=\langle h^n \rangle$ and $Y$ is the cylinder $C_H$ of $H$.

\begin{prop}\label{P: small cancellation acylindricity}
The action of $\q G$ on $\q X$ is $(\q N,\q \kappa)$-acylindrical, where 
\begin{equation*}
\q N \leq \max\left\{N, \frac{3\pi\sinh\rho}{\tau(G,X)}+1\right\}
\quad \text{and} \quad
\q \kappa = \max \{A(G,X),\kappa\} + 5\pi \sinh (150\bar \delta).
\end{equation*}
\end{prop}

\begin{rem}
	It is already known that if $G$ acts acylindrically on $X$, then so does $\q G$ on $\q X$, see Dahmani-Guirardel-Osin \cite[Proposition 2.17, 5.33]{dahmani_hyperbolically_2017}.
	However in their proof $\q\kappa$ is much larger than $\rho$.
	For our purpose we need a sharper control on the acylindricity parameters.
	With our statement, we will be able to ensure that $\q\kappa \ll \rho$.
	
	Later we will use this statement during an induction process for which we also need to control \emph{uniformly} the value of $N$.
	Unlike in \cite{dahmani_hyperbolically_2017}, if $N$ is very large, our estimates tells us that $\bar N \leq N$.
\end{rem}

\begin{proof}
Let $\q S \subset \q G$, let 
$$\q Z= \Fix(\q S, 100\bar\delta)$$ 
and let us assume that $\diam \q Z \geqslant \q \kappa$.
We are going to prove that $\q S$ contains at most $\q N$ elements.
We distinguish two cases: either $\q S$ fixes an apex $\q v \in \q {\mathcal V}$ or not.

\begin{lem}\label{L: acylindric quotient 2} If there is $\q v \in \q {\mathcal V}$, such that $\q S \subset \stab{\q v}$, then $|\q S|\leqslant 3\pi \sinh\rho/\tau(G,X) +1$.
\end{lem}
\begin{proof} If $\q S \subset \stab{\q v}$, then $\q v\in \q Z$. 
As $\diam (\q Z)\geqslant \q \kappa$, there is a point $\q x \in \q Z$ such that $|\q v - \q x| \geq \bar \kappa - \bar \delta$.
Recall that $\bar \kappa > 100 \bar \delta$.
Denote by $\q z$ the point on the geodesic $[\q v, \q x]$ at distance $100\q \delta$ from $\q v$, so that $\q z \in B(\q v, \rho/2)$.
Since $\q Z$ is $10\bar \delta$-quasi-convex, $\q z$ lies in the the $10\bar \delta$-neighborhood of $\q Z$.
In particular, for all  $\q s \in \q S$, we have $|\q s \,\q z - \q z|\leqslant 120\q \delta$. 
Let $v$ be a pre-image of $\q v$ and $z$ a pre-image of $\q z$ in the ball $B(v, \rho/2)$.
For every $\q s \in \q S$, we choose a pre-image $s \in G$ such that $|sz-z|_{\cX}\leqslant120 \q \delta$ and write $S$ for the set of all pre-images obtained in this way. 
Observe that by the triangle inequality, $|sv - v|_{\cX} \leq \rho + 120\q\delta$, for every $s \in S$.
However any two distinct apices in $\cX$ are at a distance at least $2\rho$.
Thus $S$ is contained in $\stab v$.
If $(H,Y) \in \mathcal Q$ is such that $v$ is the apex of the cone $Z(Y)$, then, by Lemma \ref{L: metric cone off 2}, $|sz-z|_{Z(Y)}\leqslant 120\q \delta < |z - v|_{Z(Y)} + |sz - v|_{Z(Y)}$. 
Let $y$ be a radial projection of $z$ on $Y$. 
By the very definition of the metric on $Z(Y)$, we get that $|sy-y|< \pi \sinh\rho$. 
Recall that every elementary subgroup is cyclic, in particular so is $\stab Y$.
Consequently the number of elements $g \in \stab Y$ such that $|gy - y| \leq r$ is linear in $r$.
More precisely, using Lemma~\ref{L: invariant cylinder}, we have
\begin{equation*}
	|S|\leqslant \frac{2(\pi \sinh\rho +112\delta)}{\tau(G,X)} + 1\leqslant \frac{3 \pi \sinh\rho}{\tau(G,X)}+1,
\end{equation*}
which yields the claim. 
\end{proof}

\begin{lem}\label{L: acylindric quotient 4}  If  $\q S$ does not  stabilise any $\q v \in \q{\mathcal V}$, then $|\q S|\leqslant \q N$.
\end{lem}

\begin{proof}
	By Lemma \ref{L: rotation quotient space},  $\q Z \cap  B(\q v, \rho- 100 \bar \delta)= \emptyset$, for every $\q v \in \q {\mathcal V}$.
 	By Lemma \ref{L: convexity fixpoints}, $\q Z$ is $10\bar \delta$-quasi-convex.  
	By Lemma \ref{L: lifts of quasi-convex},  there exists pre-images ${Z} \subset \cX$ and $S\subset G$ such that $\diam({Z})>\q \kappa$ and for all $s\in S$ and all $z\in Z$, we have $|sz-z|_{\cX}\leqslant 100 \bar \delta$. 
	
	Let us write $d=\pi \sinh (150\bar \delta)$.
	We now focus on the subset $\Fix(S,d) \subset X$.
	Let $x,y \in Z$ such that $|x-y|>\q \kappa$. 
	Let $p$, $q$ be projections of $x$, $y$ in $X$. Then, as $|p-x|_{\cX}\leqslant 100 \bar \delta$ and $ |q-y|_{\cX} \leqslant 100 \bar \delta$, $|p-q|_{\cX}\geqslant \q \kappa - 200 \bar \delta$. As $|p-q|_X\geqslant |p-q|_{\cX},$
 the distance $ |p-q|_X \geqslant\q \kappa - 200 \bar \delta.$ 
	On the other hand, $\mu (|sp-p|_X)\leqslant |sp-p|_{\cX} < 300 \bar \delta <2\rho$. Thus, by Proposition  \ref{P: mu}, $  |sp-p|_X< d.$ Similarly,  $|sq-q|_X < d$.
	This means that the diameter of $\Fix(S,d) \subset X$ is at least $\q\kappa -200 \bar \delta$, hence, larger than $A(G,X) +4d+209\delta$. 
	It follows by Proposition \ref{P: margulis 2} that $S$ generates an elementary subgroup $E$. 

	Suppose first that this subgroup $E$ is loxodromic.
	It is infinite cyclic by assumption.
	Recall that the translation length of any element in $S$ is at most $d$.
	Hence, as previously we get
	\begin{equation*}
	|S|\leqslant \frac {2(d+112\delta)}{\tau(G, X)} +1 \leqslant  \frac{3 \pi \sinh\rho}{\tau(G, X)}+1.
	\end{equation*}
	Suppose now that $E$ is an elliptic subgroup.
	In particular, the set $\Fix(S,14\delta) \subset X$ is non-empty, and by Lemma~\ref{L: convexity fixpoints}, $\Fix(S,d)$ is contained in the $d/2$-neighborhood of $\Fix(S, 14\delta)$.
In particular the diameter of $\Fix(S, 14\delta)$ is larger that $\q\kappa -200 \bar\delta-d$, hence, larger than $\kappa$. 
Consequently by acylindricity, $|S| \leq N$.
\end{proof}

This completes the proof of Proposition~\ref{P: small cancellation acylindricity}. 
\end{proof}

\subsection{\texorpdfstring{$\ell^{\infty}$-energy}{Energy}}

In this section we compare the $\ell^{\infty}$-energy of finite subset $U\subset G$ and its image $\q U\subset \q G$ respectively.

\begin{prop}\label{P: energy lift} 
Let $\q U \subset \qG$ be a finite set such that $\lambda(\q U)\leqslant \rho/5$.
If, for all $\q v\in \q {\mathcal V}$, the set $\q U$ is not contained in $\stab{\q v}$, then there is a pre-image $U\subset G $ of $\q U$ of energy 
$\lambda(U) \leqslant \pi \sinh \lambda(\q U).$
\end{prop}

\begin{proof}
Let $\epsilon > 0$.
Let $\q q \in \q X$ such that $\lambda(\q U, \q q) \leq \lambda(\q U) + \epsilon$. 
By Lemma \ref{L: rotation quotient space}, $|\q q - \q v|> \rho - (\lambda(\q U)+\epsilon)/2>4\rho/5$, for all $v \in \mathcal V$.
Let $q$ be a pre-image of $\q q$ in $\cX$. 
We choose  a pre-image $U \subset G$ of $\q U$ such that for every $u \in U$, we have $|uq-q|_{\cX}=|\q u \q q -\q q|$. 
Let $x\in X$ be a projection of $q$ onto $X$. We note that $\mu (|ux-x|_X)\leqslant |ux-x|_{\cX}\leqslant 2(\lambda(\q U) +\epsilon)<2\rho$.
Thus $|ux-x|_X\leqslant \pi \sinh (\lambda(\q U)+\epsilon)$, see Proposition \ref{P: mu}.
We just proved that $\lambda(U) \leq \pi \sinh (\lambda(\q U)+\epsilon)$ for every $\epsilon > 0$, whence the result.
\end{proof}

\section{Product set growth in Burnside groups of odd exponent}

We finally prove Theorem \ref{T: intro - large powers free burnside}. 

\subsection{The induction step} 

We will use the following.

\begin{prop}[cf. Proposition 6.18 of \cite{coulon_geometry_2014}]\label{P: induction step} 
There are distances  $\rho_0,\delta_1> 0$, and $A_0 \in  [50 \cdot 10^3\delta_1, \rho_0/500]$, as well as natural numbers $L_0$ and $n_0$ such that the following holds. 

Let $n_1 \geq n_0$ and $n \geq n_1$ be an odd integer.
Let $G$ act properly co-compactly by isometries on a proper geodesic $\delta_1$-hyperbolic space $X$ such that 
\begin{enumerate}
\item \label{enu: induction step - elem sg}
the elementary subgroups of $G$ are cyclic or finite of odd order $n$, 
\item \label{enu: induction step - params}
 $A(G,X)\leqslant A_0$ and $\tau(G,X)\geqslant \sqrt{\rho_0 L_0 \delta_1/4n_1}$, and
\item \label{enu: induction step - acyl}
 the action of $G$ is $(N,A_0)$-acylindrical, for some integer $N$.
\end{enumerate}
Let $P$ be the set of primitive hyperbolic elements $h$ of translation length $\|h\|\leqslant L_0 \delta_1$. Let $K$ be the normal closure of the set $\{h^n\mid h\in P\}$ in $G$. 

Then there is proper geodesic $\delta_1$-hyperbolic space $\overline{X}$ on which $\overline{G}=G/K$ acts properly co-compactly by isometries. Moreover,
\begin{itemize}
	\item (\ref{enu: induction step - elem sg}) and (\ref{enu: induction step - params}) hold for the action of $\overline{G}$ on $\overline{X}$;
	\item the action of $\q G$ on $\q X$ is $(\q N, A_0)$-acylindrical where $\q N = \max\left\{N, n_1\right\}$;
	\item if $\q U$ is a subset of $\q G$ with $\lambda(\q U) \leq \rho_0/5$ that does not generated a finite subgroup, then there exists a pre-image $U \subset G$ of $\q U$ such that $\lambda(U) \leq \sqrt{n_1} \sinh \lambda(\q U)$.
\end{itemize}
\end{prop}

\begin{rem}
	Note that Assumptions~(\ref{enu: induction step - params}) and (\ref{enu: induction step - acyl}) are somewhat redundant.
	Indeed, if the action of $G$ on $X$ is $(N,\kappa)$-acylindrical, then the parameters $A(G,X)$ and $\tau(G,X)$ can be estimated in terms of $\delta$, $N$ and $\kappa$ only.
	However, we chose to keep them both, to make it easier to apply existing results in the literature.
\end{rem}

\begin{proof}
This  is essentially Proposition~7.1 of \cite{coulon_geometry_2014}.  The only additional observation is point (\ref{enu: induction step - acyl}).  For details of the proof, we refer the reader to \cite{coulon_geometry_2014}.
Here, we only give a rough idea of the proof and fix some notation for later use. 

We choose for $\delta_0$, $\Delta_0$, $\delta_1$, and $\rho_0$ the constants given by the Small Cancellation Theorem, see Theorem~\ref{P: hyperbolicity quotient}.
We fix 
\begin{equation*}
	A_0 = \max \left\{6\pi \sinh(2L_0\delta_1), 50\cdot 10^3\delta_1\right\}.
\end{equation*}
Without loss of generality we can assume that $\delta_0, \Delta_0 \ll \delta_1$ while $\rho_0 \gg L_0\delta_1$.
In particular $A_0 \leq \rho_0/500$.
Following \cite[page 319]{coulon_geometry_2014}, we define a rescaling constant as follows.  Let  $$\varepsilon_n= \frac{8 \pi \sinh \rho_0 }{\sqrt{\rho_0 L_0 \delta_1}} \frac{1}{\sqrt{n}}.$$
We note for later use that if $\rho_0$ is sufficiently large (which we assume here) we have $\varepsilon_n \geq 1/\sqrt n$, for every $n > 0$.
We then choose $n_0$  such that for all $n \geq n_0$, the following holds
\begin{align}
	\varepsilon_n \delta_1 & \leq \delta_0, \\
	\label{eqn: def epsilon}
	\varepsilon_n (A_0 + 118\delta_1)  & \leq \min\{ \Delta_0, \pi \sinh(2 L_0\delta_1)\}, \\
	\frac{\varepsilon_n \rho_0L_0 \delta_1}{16 \pi \sinh \rho_0} & \leq \delta_1,\\
	\varepsilon_n &< 1.
\end{align}
These are the same conditions as in \cite[page 319]{coulon_geometry_2014} (in this reference, $\varepsilon$ is denoted by $\lambda$).  
We now fix $n_1 \geq n_0$ and an odd integer $n \geq n_1$.
For simplicity we let $\varepsilon = \varepsilon_{n_1}$.
Moreover, let 
\begin{equation*}
	\mathcal Q=\left\{\left(\left<h^n\right>,C_{E(h)}\right)\mid h\in P\right\}.
\end{equation*}
As explained in \cite[Lemma~7.2]{coulon_geometry_2014}, the small cancellation hypothesis needed to apply Theorem~\ref{P: hyperbolicity quotient} are satisfied by $\mathcal Q$ for the action of $G$ on $\varepsilon X$. 
We let $\q G$ and $\q {X}$ as in Section~\ref{sec: quotient space} (applied to $G$ acting on $\varepsilon X$).
Observe, for later use, that the map 
\begin{equation*}
	X \to \varepsilon X \xrightarrow\zeta  \q X
\end{equation*}
is $\varepsilon$-Lipschitz.
Assertions (\ref{enu: induction step - elem sg}) and (\ref{enu: induction step - params}) follows from Lemmas~7.3 and 7.4 in \cite{coulon_geometry_2014}.
By Proposition~\ref{P: small cancellation acylindricity} the action of $\q G$ on $\q X$ is $(\bar N, \bar \kappa)$-acylindrical where
\begin{equation*}
\q N \leq \max\left\{N, \frac{3\pi\sinh\rho}{\tau(G,\varepsilon X)}+1\right\}
\quad \text{and} \quad
\q \kappa = \max \{A(G, \varepsilon X), \varepsilon A_0\} + 5\pi \sinh (150\delta_1).
\end{equation*}
It follows from the definition of $\varepsilon$ and our hypothesis on $\tau(G,X)$ that $\bar N \leq \max\{ N, n_1\}$.
On the other hand by (\ref{eqn: def epsilon}) we have
\begin{equation*}
	\q \kappa \leq \varepsilon A_0 + 5\pi \sinh (150\delta_1) \leq A_0.
\end{equation*}
Hence the action of the $\q G$ on $\q X$ is $(\bar N, A_0)$-acylindrical as we announced.

Consider now a subset $\q U$ of $\q G$ such that $\lambda(\q U) \leq \rho_0 /5$ and $\q U$ does not generate a finite subgroup.
Hence, applying Proposition~\ref{P: energy lift}, we see that there exists a pre-image $U \subset G$ of $\q U$ such that the $\ell^{\infty}$-energy of $U$ for the action of $G$ on $\varepsilon X$ is bounded above by $\pi \sinh \lambda(\q U)$. 
Thus, for the action of $G$ on $X$, we obtain that 
$$\lambda (U)\leqslant \varepsilon^{-1} \pi \sinh \lambda(\q U) < \sqrt{n_1} \sinh \lambda(\q U). $$ 
 This is the lifting property stated at the end of Proposition~\ref{P: induction step}. 
\end{proof}

Assume now that $G$ is a non-elementary, torsion-free hyperbolic group.
Proposition~\ref{P: induction step} can be used as the induction step to build from $G$ a sequence of hyperbolic groups $(G_i)$ that converges to the infinite periodic quotient $G/G^n$, provided $n$ is a sufficiently large odd exponent.
For our purpose, we need a sufficient condition to detect whenever an element $g \in G$ has a trivial image in $G/G^n$.
This is the goal of the next statement, see \cite[Theorem~4.13]{coulon_detecting_2018}.
The result is reminiscence of the key argument used by Ol'shanski\u\i\ in \cite[§10]{olshanskii_periodic_1991}.
Recall that the definition of containing a (large) power (Definition~\ref{def: power-free}) involves the choice of a basepoint $p \in X$.

\begin{thm}
\label{T: non-triviality of Burnside groups} 
	Let $G$ be a non-elementary torsion-free group acting properly co-compactly by isometries on a hyperbolic geodesic space $X$.
	We fix a basepoint $p \in X$.
	There are $n_0$ and $\xi$ such that for all odd integers $n\geqslant n_0$ the following holds.
	If $g_1$ and $g_2$ are two elements of $G$ whose images in $G/G^n$ coincide, then one of them contains a $(n/2-\xi)$-power. \qed
\end{thm}

Here, we need a stronger result.
Indeed we will have to apply this criterion for any group $(G_i)$ approximating $G/G^n$.
In particular we need to make sure that the critical exponent $n_0$ appearing in Theorem~\ref{T: non-triviality of Burnside groups} does not depend on $i$.
For this reason, we use instead the following statement.

\begin{thm}
\label{T: uniform non-triviality of Burnside groups} 
	There are distances  $\rho_0,\delta_1> 0$, and $A_0 \in  [50 \cdot 10^3\delta_1, \rho_0/500]$, as well as natural numbers $L_0$, $n_0$ such that the following holds. 

	Let $n_1 \geq n_0$ and set $\xi = n_1+ 1$.
	Fix an odd integer $n \geq \max\{100, 50n_1\}$.
	Let $G$ be a group acting properly, co-compactly by isometries on a proper, geodesic, $\delta_1$-hyperbolic space $X$ with a basepoint $p \in X$, such that 
	\begin{enumerate}
	\item 
	the elementary subgroups of $G$ are cyclic or finite of odd order $n$, 
	\item 
	 $A(G,X)\leqslant A_0$ and $\tau(G,X)\geqslant \sqrt{\rho_0 L_0 \delta_1/4n_1}$.
	\end{enumerate}
	If $g_1$ and $g_2$ are two elements of $G$ whose images in $G/G^n$ coincide, then one of them contains a $(n/2-\xi)$-power.
\end{thm}

\begin{rem}
	The ``novelty'' of Theorem~\ref{T: uniform non-triviality of Burnside groups} compared to Theorem~\ref{T: non-triviality of Burnside groups} is that the critical exponent $n_0$ does not depend on $G$ but only on the parameters of the action of $G$ on $X$ (acylindricity, injectivity radius, etc).
	Note that the critical exponent given by Ol'shanski\u\i\ in \cite{olshanskii_periodic_1991} only depends on the hyperbolicity constant of the \emph{Cayley graph} of $G$.
However this parameter will explode along the sequence $(G_i)$.
Thus we cannot formaly apply this result.
Although it is certainly possible to adapt Ol'shanski\u\i's method, we rely here on the material of \cite{coulon_detecting_2018}.
\end{rem}

\begin{proof}[Sketch of proof]
	The arguments follow verbatim the ones of \cite[Section~4]{coulon_detecting_2018}.
	Observe first that the parameters $\delta_1$, $L_0$, $\rho_0$, $A_0$ and $n_0$ in \cite[p.~797]{coulon_detecting_2018} are chosen in a similar way as we did in the proof of Proposition~\ref{P: induction step} (note that the rescaling parameter that denote $\varepsilon_n$ is called $\lambda_n$ there).
	Once $n_1 \geq n_0$ has been fixed, we set, exactly as in \cite[p.~797]{coulon_detecting_2018},  $\xi = n_1 + 1$ and $n_2 = \max\{100, 50n_1\}$.
	We now fix an odd integer $n \geq n_2$.
	At this point in the proof of \cite{coulon_detecting_2018} one chooses a non-elementary torsion-free group $G$ acting properly co-compactly on a  hyperbolic space $X$ with a basepoint $p \in X$.
	Note in particular that the base point $p$ is chosen after fixing all the other parameters.
	Next one uses an analogue of Proposition~\ref{P: induction step} to build a sequence of hyperbolic groups $(G_i)$ converging to $G/G^n$.
	The final statement, that is Theorem~\ref{T: non-triviality of Burnside groups}, is then proved using an induction on $i$, see \cite[Proposition~4.6]{coulon_detecting_2018}.
	
	Observe that the fact that $G$ is torsion-free is not necessary here.
	We only need that the initial group $G$ satisfies the induction hypothesis, that is:
	\begin{enumerate}
		\item $X$ is a geodesic $\delta_1$-hyperbolic space on which $G$ acts properly co-compactly by isometries.
		\item the elementary subgroups of $G$ are cyclic or finite of odd order $n$, 
		\item $A(G,X)\leqslant A_0$ and $\tau(G,X)\geqslant \sqrt{\rho_0 L_0 \delta_1/4n_1}$.
	\end{enumerate}
	These are exactly the assumptions stated in Theorem~\ref{T: uniform non-triviality of Burnside groups}.
	In particular, we can build as in \cite{coulon_detecting_2018} a sequence of hyperbolic $(G_i)$ converging to $G/G^n$.
	The theorem is proved using an induction on $i$ just as in \cite{coulon_detecting_2018}.
	Actually the proof is even easier, since we only need a sufficient condition to detect elements of $G$ which are not trivial in $G/G^n$, while \cite{coulon_detecting_2018} provides a sufficient and necessary condition for this property.
\end{proof}

\subsection{The approximating sequence}
Let $G$ be a non-elementary torsion-free hyperbolic group.
The periodic quotient $G/G^n$ is the direct limit of a sequence of infinite hyperbolic groups $G_i$ that can be recursively constructed as follows.  
We let $\delta_1$, $\rho_0$, $L_0$, $n_0$, and $A_0 \geq 50\cdot 10^3\delta_1$ be the parameters given by Proposition \ref{P: induction step}. 

Let $G_0=G$ and let $X_0$ be its Cayley graph.
Up to rescaling $X_0$ we can assume that $X_0$ is a $\delta_1$-hyperbolic metric geodesic space and $A(G_0, X_0) \leq A_0$. 
We choose $n_1 \geq n_0$ such that  
\begin{equation*}
	\tau(G_0, X_0) \geq  \sqrt{\frac{\rho_0 L_0 \delta_1}{4n_1}}.
\end{equation*}
Recall that the action of $G_0$ on $X_0$ is proper and co-compact.
Thus there exists $N \geq n_1$, such that every subset $S \subset G_0$ for which $\Fix(S,100\delta_1)$ is non-empty contains at most $N$ elements.
Consequently the action is $(N, A_0)$-acylindrical. 
For simplicity we let $\lambda_0 = \sqrt{n_1} \pi \sinh\left( 100A_0\right)$ and denote by $m_2 = m_2(\delta_1, N, A_0, \lambda_0)$ the parameter given by Corollary~\ref{T: main acylindrical power free}.
In addition, we set $\xi = n_1 + 1$ and 
\begin{equation*}
	n_2 = \max\{ 100, 50n_1, 2(m_2 + \xi)\}.
\end{equation*}

Let $n \geq n_2$ be an odd integer.
It follows from our choices that the assumptions of Proposition~\ref{P: induction step} are then satisfied for the action of $G_0$ on $X_0$. 
 
Let us suppose that $G_i$ is already given, and acts on a $\delta_1$-hyperbolic space $X_i$ such that the assumptions of Proposition \ref{P: induction step} are satisfied. Then $G_{i+1}=\q {G_{i}}$ and $X_{i+1}= \q {X_i}$ are given by Proposition  \ref{P: induction step}.
In particular, the action of $G_{i+1}$ on $X_{i+1}$ is $(\q N, A_0)$-acylindrical, with $\q N = \max\{ N , n_1\}$.
However we chose $N \geq n_1$.
Hence the action of  $G_{i+1}$ on $X_{i+1}$ is $(N,A_0)$-acylindrical.
It follows from the construction that $G/G^n$ is the direct limit of the sequence $(G_i)$. Compare with \cite[Theorem~7.7]{coulon_geometry_2014}.

\begin{rem} As the quotient $G/G^n$ is a direct limit of non-elementary hyperbolic groups, it is an infinite group itself. In fact, for the same reason, it is not finitely presented either, see \cite[Theorem~7.7]{coulon_geometry_2014}.
\end{rem}

\subsection{Growth estimates}
\label{sec: growth estimates}
As before, we write $\varepsilon = \varepsilon_{n_1}$ for the renormalisation parameter that we used in the proof of Proposition~\ref{P: induction step}.
The action of $G_0$ on $X_0$ is proper and co-compact, hence there exists an integer $M_0$ such that for every $x \in X_0$,
\begin{equation*}
	\left| \left\{ g \in G_0 \mid \ |gx - x| \leq100 A_0\right\}\right| \leq M_0.
\end{equation*}
We now let 
\begin{equation*}
	M = \max \left\{ 
		M_0, 
		\frac{4A_0N}{\delta_1}, 
		\frac{4\cdot 10^6 N \lambda_0}{\delta_1} 
	\right\},
	\quad \text{and} \quad
	a = \frac 1M.
\end{equation*}

Let $V\subset G/G^n$ be finite and not contained in a finite subgroup. 
Recall that if $U_i \subset G_i$ is a pre-image of $V$, its energy measured in $X_i$ is defined by 
\begin{equation*}
	\lambda(U_i)= \inf_{x\in X_i} \max_{u\in U_i} |ux-x|_{X_i}.
\end{equation*}
We now let
\begin{equation}
\label{eqn: choose j}
j=\inf\left\{ i \in \mathbb N \mid \hbox{ there is a pre-image $U \subset G_i$ of $V$ such that } \lambda(U)\leqslant 100 A_0\right\}. 
\end{equation}
Recall that the map $X_i \to X_{i+1}$ is $\varepsilon$-Lipschitz.
Hence, if $U_{i+1} \subset G_{i+1}$ is the image of a subset $U_i \subset G_i$, we have $\lambda(U_{i+1}) \leqslant \varepsilon \lambda (U_i)$.
Since $\varepsilon < 1$, the index $j$ is well-defined.
Let us fix a pre-image $U_j$ of $V$ in $G_j$ such that $\lambda(U_j)\leqslant 100 A_0$.
We now distinguish two cases.

{\bfseries Case 1.} 
Assume that $j = 0$.
It follows from our choice of $M_0$, that $|V | \leq |U_0| \leq M_0$.
Thus for every $r \geq 0$ we have
\begin{equation*}
	|V^r| 
	\geq 1 
	\geq \left(\frac 1{M_0} |V|\right)^{[(r+1)/2]}
	\geq \left(a |V|\right)^{[(r+1)/2]}.
\end{equation*}

{\bfseries Case 2.} 
Assume that $j > 0$.
Note that $U_j$ cannot generate a finite subgroup $G_j$, otherwise so would $V$ in $G/G^n$.
Recall that $100 A_0 \leq \rho_0/5$.
By Proposition~\ref{P: induction step}, there exists a pre-image $U_{j-1}\subset G_{j-1}$ of $U_j$ such that the energy of $U_{j-1}$ satisfies $\lambda(U_{j-1}) \leqslant \lambda_0$.
By definition of $j$, we also have $\lambda(U_{j-1})> 100A_0$.
For simplicity we let $m = n/2 - \xi$.
It follows from our choice of $n$ that $m \geq m_2$.
Hence we can apply Corollary~\ref{T: main acylindrical power free} so that one of the following holds.
\begin{itemize}
	\item The cardinality of $U_{j-1}$ is at most $\max\{4A_0N/\delta_1,4\cdot 10^6 N \lambda_0/\delta_1\}$, which is by definition bounded above by $M$.
	In particular, the same holds for $V$ and we prove as in Case~1 that for every $r \geq 0$,
	\begin{equation*}
		|V^r| \geq \left(a |V|\right)^{[(r+1)/2]}.
	\end{equation*}
	\item There is $v \in U_{j-1}$ such that if $K(m,r)$ stands for the set of all $m$-power elements in $(U_{j-1}v)^r$, then 
	\begin{equation*}
		|K(m,r)| 
		\geq \left( \frac{1}{4\cdot 10^6 N}\frac{\delta_1}{\lambda(U_{j-1})} |U_{j-1}|\right)^r.
	\end{equation*}
	On the one hand $\lambda(U_{j-1}) \leqslant \lambda_0$.
	On the other hand, $M \geq 4\cdot 10^6 N \lambda_0/\delta_1$, while $a = 1/M$.
	Consequently
	\begin{equation*}
		|K(m,r)| 
		\geq \left( \frac{1}{4\cdot 10^6 N}\frac{\delta_1}{\lambda_0} |U_{j-1}|\right)^r
		\geq \left(a|V|\right)^r.
	\end{equation*}
	According to our choice of $n$, we have $n \geq \max\{100, 50n_1\}$.
	Moreover, by construction $G_{j-1}$ satisfies the assumptions of Theorem~\ref{T: uniform non-triviality of Burnside groups}.
	Hence, the projection $\pi \colon G_{j-1} \to G/G^n$ induces an embedding from $K(m,r)$ into $(V \pi(v))^r$.
	Consequently $|(V \pi(v))^r| \geq \left(a |V|\right)^{r}.$
	
	The proof now goes as in Corollary~\ref{T: acylindrical}.
	Let $s \geq 0$ be an integer.
	On the one hand, $(V\pi(v))^s$ is contained in $V^{2s}$, hence $|V^{2s}| \geq |(V\pi(v))^s|$.
	On the other hand $(V\pi(v))^sV$ is contained in $V^{2s+1}$.
	Right multiplication by $\pi(v)$ induces a bijection from $G/G^n$ to itself.
	Hence
	\begin{equation*}
		|V^{2s+1}| \geq |(V\pi(v))^sV| = |(V\pi(v))^{s+1}| \geq \left(a |V|\right)^{s+1}.
	\end{equation*}
	It follows that $|V^r| \geq \left(a |V|\right)^{[(r+1)/2]}$, for every integer $r \geq 0$.
\end{itemize}	
This completes the proof of Theorem \ref{T: intro - large powers free burnside}.

\begin{proof}[Proof of Corollary~\ref{C: intro - large powers free burnside}]
Let $n_0>0$ and $a>0$ be the constants given by Theorem~\ref{T: intro - large powers free burnside}.  
We fix $N$ such that $a^3N>1$.
Let $n\geqslant n_0$.
Let us take a subset $V\subset G/G^n$ that is not contained in a finite subgroup and that contains the identity. Then, for all $k\geqslant 1$, we have $V^{k-1} \subseteq V^k$. As $V$ is not contained in a finite subgroup, this implies that $|V^k|>|V^{k-1}|$. Thus $a^3|V^ N|>1$.
We now apply twice Theorem~\ref{T: intro - large powers free burnside},
first with the set $V^{3N}$, and second with $V^N$.
For every integer $r \geq 0$, we have
\begin{equation*}
	\left|V^{3rN} \right| 
	\geq \left(a \left|V^{3N}\right| \right)^{[(r+1)/2]}
	\geq \left(a \left(a\left|V^N\right|\right)^2 \right)^{[(r+1)/2]}
	\geq \left(a^3|V^N| \cdot |V^N| \right)^{[(r+1)/2]}
\end{equation*}
Recall that $a^3|V^ N|>1$.
Hence, for every integer $r \geq 0$,
\begin{equation*}
	\left|V^{3rN} \right| 
	\geq |V^N| ^{[(r+1)/2]}
	\geq |V|^{[(r+1)/2]}.
\end{equation*}
Taking the logarithm and passing to the limit we get
\begin{equation*}
	h(V)
	\geq \frac 1{6N} \ln(|V|).
\end{equation*}
Since $V$ does not lie in a cyclic subgroup and contains the identity, it has at least three elements, whence the second inequality in our statement.
\end{proof} 


\addtocontents{toc}{\setcounter{tocdepth}{-10}}
\bibliographystyle{alpha}

\bibliography{product}

\end{document}